%% file: LILGvE.tex
\documentclass[12pt]{amsart}
\usepackage{amsmath,amsfonts,amssymb,latexsym}
\usepackage[all]{xy}
\usepackage{pb-diagram}

\usepackage{mathrsfs}

\theoremstyle{plain}
        \newtheorem{theorem}{Theorem}[section]
        \newtheorem{lemma}[theorem]{Lemma}
        \newtheorem{proposition}[theorem]{Proposition}
        \newtheorem{corollary}[theorem]{Corollary}
        
        \theoremstyle{definition}
        \newtheorem{definition}[theorem]{Definition}
        \newtheorem{remark}[theorem]{Remark}

\newtheorem*{thma}{Theorem A}
\newtheorem*{thmb}{Theorem B}

\newcommand{\N}{{\mathbb N}}

\newcommand{\C}{{\mathbb C}}
\newcommand{\R}{{\mathbb R}}

\newcommand{\mbA}{{\mathbb A}}

\newcommand{\sfG}{\mathsf{G}}

\newcommand{\supp}{\operatorname{supp}}

\newcommand{\Ch}{\operatorname{Ch}}
\newcommand{\Hom}{\operatorname{Hom}}
\newcommand{\symb}{\operatorname{Sym}}
\newcommand{\asymb}{\operatorname{ASym}}
\newcommand{\jsymb}{\operatorname{JSym}}
\newcommand{\symm}{\mathsf{S}}
\newcommand{\op}{\operatorname{Op}}
\newcommand{\tr}{\operatorname{tr}}
\newcommand{\lb}{\left<}
\newcommand{\rb}{\right>}
\newcommand{\Op}{\operatorname{Op}}
\newcommand{\End}{\operatorname{End}}
\newcommand{\Ind}{\operatorname{Ind}}
\newcommand{\Tot}{\operatorname{Tot}}
\newcommand{\Mat}{\operatorname{Mat}}

\newcommand{\calW}{\mathcal{W}}
\newcommand{\calE}{\mathcal{E}}
\newcommand{\calD}{\mathcal{D}}
\newcommand {\calA}{\mathcal{A}}
\newcommand{\calF}{\mathcal{F}}

\newcommand{\U}{\mathscr{U}}

\newcommand{\calU}{\mathcal{U}}

\newcommand{\calC}{\mathcal{C}}

\newcommand{\llangle}{\langle\!\langle}
\newcommand{\rrangle}{\rangle\!\rangle}

\title[The localized longitudinal index theorem]{The localized longitudinal index theorem for Lie groupoids and the van Est map}
\author{M.J.~Pflaum, H. Posthuma,~\textrm{and} X.~Tang}
\address{\newline
Markus J. Pflaum, {\tt markus.pflaum@colorado.edu}\newline
         \indent {\rm Department of Mathematics, University of Colorado,
         Boulder, USA}
         \newline
        Hessel Posthuma, {\tt H.B.Posthuma@uva.nl}\newline
         \indent {\rm Korteweg-de Vries Institute for Mathematics,
        University of Amsterdam,
         The Netherlands} 
         \newline
        Xiang Tang, {\tt xtang@math.wustl.edu}   \newline
         \indent {\rm  Department of Mathematics, Washington University,
         St.~Louis, USA}}
\begin{document}
\maketitle
\begin{abstract}
We define the ``localized index'' of longitudinal elliptic operators on Lie groupoids associated to Lie algebroid cohomology classes. 
We derive a topological expression for these numbers using the algebraic index theorem for Poisson manifolds on the dual of the Lie algebroid.
Underlying the definition and computation of the localized index, is an action of the Hopf algebroid of
jets around the unit space, and the characteristic map it induces on Lie algebroid cohomology.
This map can be globalized to differentiable groupoid cohomology, giving a definition as well as a computation of the ``global index''. 
The correspondence between the ``global'' and ``localized'' index is given by the van Est map for Lie groupoids.
\end{abstract}
\tableofcontents

\input{intro}
\section{The characteristic map and a global pairing}
\label{sec:char}

In this section we construct the basic map underlying the definition of localized indices of elliptic operators. Throughout, 
$\sfG\rightrightarrows M$ denotes a Lie groupoid with source and target map denoted by $s,t:\sfG \to M$, and the unit map denoted by $u :M \rightarrow \sfG$. 
We denote an arrow $g\in\sfG$ 
with $s(g)=x$ and $t(g)=y$ as $x\stackrel{g}{\rightarrow} y$, so that the 
multiplication $g_1g_2$ is defined when $t(g_1)=s(g_2)$. 

The Lie algebroid of $\sfG$ is  the vector bundle $u^* (\ker Tt)$ over $M$ and is denoted by $A$. 
Its space of smooth sections $\Gamma^\infty(M,A)$ 
carries a Lie algebra structure $[~,~]$, together with an anchor map $\rho:A\to TM$ satisfying
\begin{align*}
\rho([X,Y])&=[\rho(X),\rho(Y)],\\
[X,fY]&=f[X,Y]+\rho(X)(f)\cdot Y, 
\end{align*}
for $X,Y\in \Gamma^\infty (M,A)$, and $f \in \calC^\infty (M)$. 

\subsection{Differentiable cohomology}
\label{diff-coh}
The classifying space $B\sfG$ of $\sfG$ can be realized as a simplicial 
manifold by defining $\sfG_k$, $k\geq 0$ to be the manifold of $k$-tuples of composable arrows:
\[
 \sfG_k:=\{(g_1,\ldots,g_k) \in \sfG^k \mid t(g_i)=s(g_{i+1}) \text{ for $i=1,\ldots, k-1$}\},
\]
with the convention that $\sfG_1:=\sfG$ and $\sfG_0:=M$. 
The face operators $\partial_i:\sfG_k\to\sfG_{k-1}$ are given by
\begin{equation}
\label{simplicial}
\partial_i(g_1,\ldots,g_k)=\begin{cases} (g_2,\ldots,g_k), &i=0,\\
(g_1,\ldots,g_ig_{i+1},\ldots,g_k), &1\leq i\leq k-1,\\
(g_1,\ldots,g_{k-1}), &i=k,\end{cases}
\end{equation}
and for $k=1$, we put $\partial_0:=s$ and $\partial_1:=t$.

This simplicial structure induces a differential $d$ on the graded vector space 
$C^\bullet_\textup{diff} (\sfG;\C) := \bigoplus_{k\geq 0} C^k(\sfG;\C)$, where $C^k(\sfG;\C):=\calC^\infty(\sfG_k,\C)$.
More precisely, $d$ is  given by
\begin{equation}
\label{groupoid-diff}
\begin{split}
d\varphi(g_1,&\ldots,g_{k+1}):=\varphi(g_2,\ldots,g_{k+1})\\
&+\sum_{i=1}^k(-1)^k\varphi(g_1,\ldots,g_ig_{i+1},\ldots,g_k)+(-1)^{k+1}\varphi(g_1,\ldots,g_k).
\end{split}
\end{equation}
The cohomology of this complex is called the {\em differentiable groupoid cohomology}, 
denoted $H^\bullet_\textup{diff}(\sfG;\C)$.

For our considerations, it is important that there is also a cyclic structure on 
$\big(\sfG_k\big)_{k\geq 0}$
 given by 
\[
t_k(g_1,\ldots,g_k):=\left( (g_1g_2\cdots g_k)^{-1},g_1,\ldots,g_{k-1}\right),
\]
with which the classifying space is even a cyclic manifold. (Cyclic objects can be defined in arbitrary categories, cf.~\cite{loday}.)
On the level of smooth functions, this gives $\big( C^k (\sfG;\C) \big)_{k\geq 0}$ 
the structure of a cocyclic vector space by pull-back:
\[
\tau_k\varphi(g_1,\ldots,g_k):=\varphi\left((g_1g_2\cdots g_k)^{-1},g_1,\ldots,g_{k-1}\right),
\]
which turns the graded vector space $C^\bullet_\textup{diff}(\sfG;\C)$ into a cocyclic module. Let ${\mathcal {B}} C^\bullet_{\textup{diff}}(\sfG; \C)$ be the associated mixed complex. The resulting cyclic cohomology, the cohomology of $ {\rm Tot}^{\bullet}{\mathcal {B}}C_\textup{diff}(\sfG;\C)$, is easily 
computed to be given by
\[
HC^\bullet_\textup{diff}(\sfG):=\bigoplus_{i\geq 0}H^{\bullet+2i}_\textup{diff}(\sfG;\C).
\]

\subsection{Unimodularity and a trace}
\label{sec:unitr}
Define the following line bundle of ``transversal densities'' 
$L:=\bigwedge^\textup{top} T^*M\otimes\bigwedge^\textup{top}A$. It was observed in \cite{elw} that 
$L$ carries a canonical representation of $\sfG$, even though the Lie algebroid $A$ and the tangent 
bundle $TM$ do not carry such a representation. 
Following this paper, we will recall the notion of {\em unimodularity} of the groupoid $\sfG$.

Assume that $L$ is trivial as a line bundle and therefore can be trivialized by choosing a ``volume form'' $\Omega\in\Gamma^\infty(M,L)$.
\begin{definition}
A Lie groupoid $\sfG$ is {\em unimodular} if there exists an invariant  global nonvanishing section $\Omega$ of $L$.
\end{definition}
The assumption of unimodularity has a cohomological origin. To see this, choose an arbitrary section $\Omega$ of $L$. For such a choice, 
there exists a smooth nonvanishing function $\delta_\sfG^\Omega\in \calC^\infty(G)$ such that
\[
  g(\Omega_{s(g)})=\delta_\sfG^\Omega(g)\Omega_{t(g)}.
\]
One easily checks that $\delta^\Omega_\sfG(g_1g_2)=\delta_{\sfG}^\Omega(g_1)\delta_\sfG^\Omega(g_2)$ for $(g_1,g_2)\in\sfG_2$, so that $\log(\delta_\sfG^\Omega)$ defines an element in $H^1_\textup{diff}(\sfG;\R)$, which is easily checked to be independent of the choice of $\Omega$. If 
this class $[\log\delta] \in H^1_\textup{diff}(\sfG;\R)$ is zero, it means that there exists a smooth function $\eta^\Omega\in \calC^\infty(M,\R_+)$ such that 
\[
  \delta_\sfG^\Omega(g)=\eta^\Omega(s(g))\eta^\Omega(t(g))^{-1}.
\]
From this, it follows that the product $(\eta^\Omega)^{-1}\Omega$ is invariant.

The {\em convolution algebra} $\calA$ of $\sfG$ is given by the space
 $\Gamma_\textup{c}^\infty\left( \sfG,s^*{\bigwedge}^\textup{top} A^* \right)$ 
of longitudinal densities equipped with the product
\begin{equation}
\label{convolution}
  (f_1*f_2)(g):=\int_{\sfG_{t(g)}} f_1(gh^{-1}) f_2(h) \, .
\end{equation}
This formula needs to be interpreted as follows. First note that for $x\in M$ the symbol
$\sfG_x$ denotes the submanifold of $\sfG$ consisting of all arrows having target $x$. 
Then observe that ${f_2}_{|\sfG_x}$ is a volume form on $\sfG_x$, since for $h\in \sfG_x$:
\[
  {s^*A^*}_h \cong \big( \ker T_h t\big)^* = T^*_h \sfG_x \, . 
\]
Moreover, $f_1 (gh^{-1}) \in \left({\bigwedge}^\textup{top} A^*\right)_{ s(g)}$ for all 
$h\in \sfG_{t(g)}$, 
hence the integral in \eqref{convolution} is well-defined, indeed. 
\begin{remark}
\label{E-valued-conv}
Note that for a $\sfG$-vector bundle $E\rightarrow M$, one can even define an $\End (E)$-valued 
convolution algebra whose underlying space is given by  
$\Gamma_\textup{c}^\infty\left( \sfG,s^*({\bigwedge}^\textup{top} A^* \otimes \End(E) \right)$ 
and whose product $*$ is given  by the formula \eqref{convolution},  provided the pointwise product also involves composition of endomorphisms. We may write $\calA(E)$ for this algebra, but most of the time in this paper we suppress $E$ from the notation and simply write $\calA$.  
\end{remark}
Given an invariant volume form $\Omega\in\Gamma^\infty(M,L)$, consider the following 
functional on $\calA$: 
\begin{equation}
\label{trace}
 \tau_\Omega:\calA \rightarrow \C , \quad  \tau_\Omega(f):=\int_M f \lrcorner \, \Omega ,
\end{equation}
where $ f \lrcorner \, \Omega $ denotes the contraction of $\Omega$ with $f$. Over the 
footpoint $x\in M$ this contraction is defined by 
\[
   f \lrcorner \, \Omega (x) := \sum \langle f(u(x)) , \Omega_{(2)} (x) \rangle \, 
   \Omega_{(1)} (x)  ,
\]
where, using a Sweedler kind of notation, 
$\Omega(x) = \sum \Omega_{(1)} (x) \otimes \Omega_{(2)} (x)$ with  
$\Omega_{(1)} (x) \in {\bigwedge}^\textup{top} T_x^* M$, and 
$\Omega_{(2)} (x) \in {\bigwedge}^\textup{top} A_x$.
\begin{proposition}
$\tau_\Omega$ defines a trace on the convolution algebra.
\end{proposition}
\begin{proof}
This is a straightforward computation.
\end{proof}

From now on we assume the groupoid $\sfG$ to be unimodular, so that we can use this trace. Only in section \ref{sec:nonunimodular} we shall remove this assumption of unimodularity and explain how the index theory works for general Lie groupoids.

\subsection{The characteristic map to cyclic cohomology}
\label{char-map}
We now consider the cyclic cohomology of the convolution algebra $\calA$ associated to $\sfG$. 
It was proved in \cite{cm} that $\calA$ is an $H$-unital algebra, so its cyclic cohomology
can be defined using the standard complexes. These result from a canonical cocyclic module 
structure on the graded vector space $X^\bullet(\calA):=\Hom_\textup{b} \left(\calA^{\otimes(\bullet+1)},\C\right)$, 
as defined in \cite{loday}, where here we use the completed inductive tensor product and consider only 
bounded functionals.

Using the trace of Proposition \ref{trace}, consider the characteristic map $\chi_\Omega:C^k_\textup{diff}(\sfG;\C)\to X^k(\calA)$ given by
\begin{equation}\label{eq:chi-uni}
\begin{split}
\chi_\Omega(\varphi_1&\otimes\ldots\otimes \varphi_k)(a_0\otimes\ldots\otimes a_k)\\
&= \tau_\Omega\left(a_0*(\varphi_1\cdot a_1)*\ldots*(\varphi_k\cdot a_k)\right)\\
&=\int_M\left(\int_{g_0\cdot \ldots \cdot g_k=1_x}\hspace{-1cm}a_0(g_0)\varphi_1(g_1)a_1(g_1)\cdots \varphi_k(g_k) a_k(g_k)\right)\Omega(x),
\end{split}
\end{equation}
where we have written $\varphi=\varphi_1\otimes\ldots\otimes\varphi_k$. This is for notational purposes only, clearly the map above is well-defined $\varphi\in C^k_{\rm diff}(\sfG)$.
\begin{proposition}\label{prop:char-map}
The characteristic map is a morphism of cocyclic modules.
\end{proposition}
\begin{proof}
Again, given the cocyclic structure on both sides, this is a straightforward computation.  
\end{proof}
On the level of cohomology, this leads to a map
\[
HC^\bullet_\textup{diff}(\sfG;\C)\to HC^\bullet(\calA).
\]
Using the Connes--Chern character $\Ch :K_0(\calA)\to HC_{\rm ev}(\calA)$, cf. \cite{loday}, together with the canonical pairing between cyclic cohomology and homology, we obtain a pairing
\begin{equation}
\label{global-pairing}
\left<~,~\right> : HC^{\rm ev}_\textup{diff}(\sfG;\C)\times K_0(\calA)\to\C.
\end{equation}
explicitly given by 
\[
\left<[\varphi],[e]\right>:=\sum_{i=0}^k (-1)^i\frac{(2i)!}{i!}\chi_\omega(\varphi_{2i})\left(\left(e-\frac{1}{2}\right)\otimes e\otimes\ldots\otimes e\right),
\]
where $\varphi=(\varphi_{2k},\ldots,\varphi_0)\in {\rm Tot}^{2k}{\mathcal {B}} C_\textup{diff}(\sfG;\C)$ and $e \in \Mat_\infty (\calA)$ 
is a projection.
\section{Localization}
\label{sec:localization}

\subsection{Localized K-theory}
In \cite{mw}, Moscovici and Wu defined a localized version of $K$-theory for a smooth manifold $M$, using the algebra of smoothing operators. This algebra of smoothing operators can be interpreted as the convolution algebra of the pair groupoid $M\times M\rightrightarrows M$. From this point of view, the definition of localized $K$-theory can naturally be generalised to the setting of arbitrary Lie groupoids as follows: for any open neighbourhood $U\subset \sfG$ of the unit space $M$, denote by $\calA(U)$ the elements in $\calA$ with support in $U$. Of course, $\calA(U)$ is not a subalgebra of $\calA$, since we have
\[
\calA(U)*\calA(U)\subset\calA\left(U_\partial\right)
\]
with $U_\partial:=\partial_1\left(\partial_0^{-1}(U)\cap \partial_2^{-1}(U)\right)$, cf.~Eq.~\eqref{simplicial}. 
Nevertheless, we can define
\[
\begin{split}
K_0(\calA(U)):=\big\{(P,e)\in &\Mat_\infty\left(\calA(U)^+\right)\times \Mat_\infty(\C)) \mid P^2=P=P^*,\\
 & e^2=e=e^*, \text{ and } P-e\in \Mat_\infty\left(\calA(U)\right)\big\}\big/ \sim.
\end{split}
\]
Here, $(P_1,e_1)\sim (P_2,e_2)$, if there exist  continuous piecewise $\calC^1$-paths 
$\varrho : [0,1] \rightarrow \Mat_\infty\left(\calA(U)^+\right)$ and 
$\eta : [0,1] \rightarrow \Mat_\infty\left( \C \right)$ of projectors in $\Mat_\infty\left(\calA(U)^+\right)$
resp.~$\Mat_\infty\left( \C \right)$ such that $\varrho (i) = P_i$ and $\eta (i) = e_i$ for $i=0,1$.

Clearly, an inclusion of open sets $V\subset U$ leads to a map $K_0(\calA(V))\to K_0(\calA(U))$, and we 
define the {\em localized $K$-theory} as the projective limit
\begin{equation}
\label{def-loc-k}
K_0^\textup{loc}(\calA):=\lim_{\longleftarrow\atop U\supset M} K_0(\calA(\calU)).
\end{equation}
By definition, this abelian group comes equipped with a canonical map
\begin{equation}
\label{forgetful}
K_0^\textup{loc}(\calA)\to K_0(\calA),\quad [P]_\textup{loc}\mapsto [P],
\end{equation}
which forgets about the support of projectors. 
\subsection{Lie algebroid cohomology}
Next, we localize the differentiable groupoid cochain complex to the unit space. We do this as follows: 
for any open neighbourhood $U$ of $M$ in $\sfG$, we define inductively
\begin{equation}
\label{ind-nghbd}
U_1 := U, \qquad U_k:=\bigcap_{i=0}^k\partial_i^{-1}(U_{k-1}),\quad k\geq 2.
\end{equation}
Obviously, $U_k$  is an open neighbourhood of the unit space $M$ in $G_k$. 
For $k=0$, we simply put $U_0=M$. Put $C^k_\textup{diff}(\sfG, U;\Bbbk):=\calC^\infty(U_k,\Bbbk)$ for
$\Bbbk =\R,\C$. 
The definition of the neighbourhoods $U_k$ has been designed to ensure that the differentiable 
groupoid cohomology differential restricts to the cochains on $U_\bullet$, provided we define
\[
df:=(s^*f-t^*f)|_U,
\] 
in degree zero. This defines a cochain complex $(C^\bullet_\textup{diff}(\sfG,U;\Bbbk),d)$ for any open neighbourhood $U$ of $M$ in $\sfG$. 
Clearly, for an inclusion $V\subset U$ of such neighbourhoods we have a canonical restriction morphism of complexes 
$C^\bullet_\textup{diff}(\sfG,U;\Bbbk)\to C^\bullet_\textup{diff}(\sfG,U;\Bbbk)$. As $U$ gets smaller and smaller around $M$, the $U_k$ 
localize around $M$ inside $\sfG_k$. We therefore see that the direct limit
\begin{equation}
\label{loc-cochain}
\hat{C}^k_\textup{diff}(\sfG;\Bbbk):=\lim_{\longrightarrow\atop U\supset M}C^k_\textup{diff}(\sfG,U;\Bbbk)
\end{equation}
consists of germs of smooth functions on $\sfG_k$ around $M$. By construction, the differentiable  groupoid differential $d$ induces 
a differential on $\hat{C}^k_\textup{diff}(\sfG;\Bbbk)$, denoted $\hat{d}$. 

Recall that {\em Lie algebroid cohomology} $H^\bullet_{\rm Lie}(A;\C)$ is defined as the cohomology of the complex $\Omega^\bullet(A)=\Gamma^\infty(M,\bigwedge^\bullet A^*)$ 
equipped with the de Rham-type differential $d:\Omega^k_\textup{Lie}(A)\to\Omega_{\rm Lie}^{k+1}(A)$
\begin{equation}
\label{lie-alg-diff}
\begin{split}
dc(X_1,\ldots,X_{k+1}):=&\sum_i(-1)^i\rho(X_i)c(X_1,\ldots,\hat{X}_i,\ldots,X_{k+1})\\
&\sum_{i<j}(-1)^{i+j+1}c([X_i,X_j],X_1,\ldots,\hat{X}_i,\ldots,\hat{X}_j,\ldots X_k).
\end{split}
\end{equation}

Now, consider the map $\Phi:\hat{C}^k_\textup{diff}(\sfG;\Bbbk)\to \Omega^k_\textup{Lie}(A,\Bbbk)$ defined by
\begin{equation}
\label{eq:q.i.-lie}
\Phi(\hat{\varphi}(X_1,\ldots,X_k):=\sum_{\sigma\in S_k}(-1)^\sigma R_{X_{\sigma(1)}}\ldots R_{X_{\sigma(k)}}\hat{\varphi},
\end{equation}
where, for $X\in\Gamma(M,A)$, $R_X\hat{\varphi}\in \hat{C}^{k-1}_\textup{diff}(\sfG;\C)$ is defined by
\[
R_X\hat{\varphi}(g_2,\ldots,g_{k}):=\frac{\partial \hat{\varphi}(-,g_2,\ldots,g_k)}{\partial X}(s(g_2)).
\]
Here $X$ is lifted to an invariant vector field on $s^{-1}(g_2)$ along which the (germ of the) function $\hat{\varphi}(-,g_2,\ldots,g_k)$ is differentiated.
\begin{proposition}
\label{loc-lie}
The map $\Phi$ defines a quasi-isomorphism of cochain complexes:
\[
H^\bullet\left(\hat{C}^k_\textup{diff}(\sfG;\Bbbk),\hat{d}\right)\cong H^\bullet_\textup{Lie}(A;\Bbbk).
\]
\end{proposition}
\begin{proof}
For the related case of $\infty$-jets instead of germs, this is proved in \cite{kp}. In fact the Koszul resolution used in that proof can be straightforwardly adapted to our setting of germs. We will later see an alternative proof in Section \ref{homogeneous}.
\end{proof}
\begin{remark}
In \cite{kp}, Lie algebroid cohomology appears as the Hopf-cyclic cohomology of a Hopf algebroid of jets associated to the Lie algebroid. In fact, the characteristic map of Section \ref{char-map}, when localized to the unit space, can be understood in this framework as coming from a Hopf algebroid module structure on $\calA$, together with the fact that the trace is invariant. We have decided in this paper to not explicate the Hopf algebroid point of view and just work with the characteristic maps, which are explicit enough for our purposes. 
\end{remark}
\subsection{The localized pairing}
Now that we have localized the elements in $K$-theory and differentiable groupoid cohomology, we can consider the localized version of the pairing \eqref{global-pairing}.
\begin{proposition}
The global pairing \eqref{global-pairing} localizes to a canonical pairing
\[
\llangle \ ,\ \rrangle : HC^\textup{ev}_\textup{Lie}(A)\times K_0^\textup{loc}(\calA)\to\C.
\]
\end{proposition}
\begin{proof}
By Proposition \ref{loc-lie}, Lie algebroid cohomology is computed by the complex given by \eqref{loc-cochain} of differentiable groupoid cochains that are localized to the unit space $M$. As a direct limit over open neighbourhoods of $M$ in $\sfG$, the pairing \eqref{global-pairing} canonically lifts to a pairing with $K^{\rm loc}_0(\calA)$, which by definition \eqref{def-loc-k} is a projective limit. This proves the proposition.
\end{proof}

By Proposition \ref{loc-lie} there is a canonical morphism $E:H^\bullet_\textup{diff}(\sfG;\C)\to H^\bullet_\textup{Lie}(A;\C)$ given on the cochain level by localizing groupoid cochains to the unit space, and possibly composing with $\Phi$ to land in the standard Chevalley--Eilenberg complex $(\Omega^\bullet_A,d)$. This map is called the van Est map \cite{wx,crainic}. 

\begin{theorem}
For $[P]_\textup{loc}\in K_0^\textup{loc}(\calA)$ and $\varphi\in HC^\textup{ev}_\textup{diff}(\sfG)$ the following equality holds true:
\[
\llangle[P]_\textup{loc},E(\varphi)\rrangle:=\left<[P],[\varphi]\right>.
\]
\end{theorem}
\begin{proof}
This is now immediately clear from the formalism we have developed so far. 
\end{proof}
We have now completed the proof of theorem A in the introduction.
\begin{remark}
The van Est map $E:H^\bullet_\textup{diff}(\sfG;\C)\to H^\bullet_\textup{Lie}(A;\C)$ is quite well understood by the work of Crainic: Theorem 4 of \cite{crainic} states that when the $s$-fibers of $\sfG$ are homologically $n$-connected, $E$ is an isomorphism in degree $\leq n$ and injective in degree $n+1$.
\end{remark}

\section{The localized index class}
\label{sec:local-pairing}

In this section, we use the longitudinal pseudodifferential calculus on Lie groupoids of \cite{nwx} to show that an 
elliptic operator in that calculus defines a localized index class in $K_0^\textup{loc}(\calA)$, 
refining the index class in $K_0(\calA)$ (cf.~also \cite{montpier}). 
\subsection{Longitudinal pseudodifferential operators}
Let us briefly recall the definition of longitudinal pseudodifferential operators from \cite{nwx}. 
Informally, these are $\sfG$-invariant families of pseudodifferential operators along the fibers of the submersion $t:\sfG\to M$. 
In the following, we write $\sfG_x,~x\in M$ for the fiber $t^{-1}(x)\subset\sfG$. 
Let $E$ be a vector bundle over $M$, and denote by $t^*E$ its pull-back to $\sfG$. Any arrow $g\in\sfG$ defines a linear operator 
$U_g:\Gamma^\infty (\sfG_{s(g)},t^*E)\to \Gamma^\infty (\sfG_{t(g)},t^*E)$ by right translation:
\[
(U_g s)(h):=s(hg^{-1}),\quad s\in \Gamma(\sfG_{s(g)},t^*E),
\]
which uses the fact that $(t^*E)_h=E_{t(h)}=(t^*E)_{hg}$.
A longitudinal pseudodifferential operator $P$ of order $m$ is given by a family $\{P_x\}_{x\in M}$ of pseudodifferential operators 
$P_x\in\Psi^m(\sfG_x;t^*E)$, which is:
\begin{itemize}
\item {\em smooth} in its dependence of $x$, in a precise sense explained in \cite[Def.~6]{nwx},
\item {\em invariant} in the sense that
\[
P_{t(g)}U_g=U_gP_{s(g)},
\]
for all $g\in\sfG$.
\end{itemize}
To turn the space of such pseudodifferential operators into an algebra, we need some condition on the support. As a family 
of operators on the $t$-fibers, the Schwartz-kernel of a pseudodifferential operator $P$ is in fact a family of kernels 
$\left( k_x\right)_{x\in M}$, where $k_x$ is a distribution on $t^{-1}(x)\times t^{-1}(x)$. The support of $P$ is now defined as
\[
\supp(P):=\overline{\cup_{x\in M}\supp(k_x)},
\]
where the closure is taken in the space 
\[
 \sfG^{\times_{\!t} 2}:=\sfG {\times_{\!t}} \sfG := 
 \left\{ (g_1,g_2)\in\sfG\times\sfG \mid t(g_1)=t(g_2) \right\} .
\]
Consider the map $\mu:\sfG^{\times_{\!t} 2}\to\sfG$ given by $\mu(g_1,g_2)=g_1g_2^{-1}$. 
A pseudodifferential operator $P$ is said to be uniformly supported if $\mu(\supp(P))\subset\sfG$ is compact.
 We write $\Psi^m(\sfG,E)$ for the space of uniformly supported longitudinal pseudodifferential operators 
on $\sfG$ of order $m$ acting on sections of the vector bundle $t^*E$. As usual, we set
\[
  \Psi^\infty(\sfG,E)=\bigcup_m\Psi^m(\sfG,E),\quad \Psi^{-\infty}(\sfG,E)=\bigcap_m\Psi^m(\sfG,E).
\]
It is proved in \cite{nwx} that composition of families of operators turns $\Psi^{\infty}(\sfG,E)$ into a 
filtered algebra. Furthermore, the subalgebra $\Psi^{-\infty}(\sfG,E)$ of smoothing operators is isomorphic, 
via the so-called ``reduced kernel'' defined by
\[
  k_P(g):=k_{t(g)}(g,t(g)),
\]
to the convolution algebra $\calA(E)$ of remark \ref{E-valued-conv}. For $E$ the trivial bundle, this is just \eqref{convolution} and we 
find the algebra $\calA$.

Next, we describe the symbol calculus for longitudinal pseudodifferential  operators. To this end recall first 
that the space $\symb^m(A^*)$  of \emph{symbols of order} $m\in \R$ on the bundle $A^*$ is defined as the space of 
smooth functions $a$ on the total space $A^*$ such that in bundle coordinates $(x,\xi)$ an estimate of the form  
\begin{equation}
  \label{eq:symb}
   \left| \big(\partial^\alpha_x\partial^\beta_\xi a\big)_{|\pi^{-1} (K)} \right| \leq C_{K,\alpha,\beta} \left( 1 + \| \xi \| \right)^{m -|\beta|}
\end{equation}
holds true for each compact $K\subset M$ in the domain of $x$, all multi-indices
$\alpha\in \N^{\dim M}$, $\beta \in \N^{\dim \sfG - \dim M}$ and some $C_{K,\alpha,\beta} >0$ depending only on $K,\alpha$, and $\beta$.
Analogous to  pseudodifferential operators, we put
\[
  \symb^\infty(A^*)=\bigcup_m\symb^m(A^*),\quad \symb^{-\infty}(A^*)=\bigcap_m\symb^m(A^*).
\]
Observe that $ \symb^m(A^*)$ contains the space $ \symb^m_\textup{cl}(A^*)$ of \emph{classical symbols}
which we will use in the following as well (cf.~\cite{shubin} for details).

Next, we choose a connection $\nabla$ on $A$. Pulling back along $s$, we obtain a connection 
$s^*\nabla$ on $s^*A$, which restricts to a usual connection on each fiber $\sfG_x$ of $t$. With this, 
we can define the exponential mapping $\exp_x:A_x\to\sfG$. Varying over $x\in M$, we get the global 
exponential mapping
\[
  \exp_\nabla:U\to\sfG,
\]
where $U\subset A$ is a neighbourhood of the zero section. This map is a diffeomorphism onto a neighbourhood 
$V$ of $M$ in $\sfG$. Choose a cut-off function $\phi\in \calC^\infty_\textup{c}(\sfG)$ with support in $V$ and 
equal to $1$ in a smaller neighbourhood of $M$. For $\xi\in A^*$, denote by $e_\xi\in \calC^\infty(V)$ the 
function defined by
\[
  e_\xi(g):=\phi(g)e^{\sqrt{-1}\left<\xi,\exp_\nabla^{-1}(g)\right>}.
\]
Then the symbol of $P$ is defined as
\[
  \sigma_{\nabla,\phi}(P)(x,\xi)=(P_xe_\xi)(x),\quad\mbox{for}~\xi\in A_x^*.
\]
This defines a map $\sigma_{\nabla,\phi}:\Psi^m(\sfG)\to\symb^m_\textup{cl}(A^*)$. 
It is independent of $\nabla$ and $\phi$ modulo $\symb^{m-1}_\textup{cl}(A^*)$, and defines the principal symbol
\begin{equation}
\label{principal-symbol}
  \sigma_m:\Psi^m(\sfG)\to\symb^m(A^*)\slash \symb^{m-1}(A^*)
\end{equation}

Conversely, let us define a quantization map $\op:\symb_\textup{cl}^m(A^*)\to \Psi^m(\sfG)$ by
\[
\left(\op(a)(f)\right)(g):=\int_{A^*_{s(g)}}\int_{\sfG_{t(g)}}e_{-\xi}(hg^{-1})a(s(g),\xi)f(h)dhd\xi.
\]
The reduced kernel of this operator then is given by
\begin{equation}
\label{reduced-kernel}
k_{\op(a)}(g)=\int_{A^*_{s(g)}}e_{-\xi}(g^{-1})a(s(g),\xi).
\end{equation}

\begin{remark}
Following our notation from \cite{ppt} we have shifted the usual normalization factors  
$\frac{1}{(2\pi\sqrt{-1})^{n/2}}$  in the  definition of $\op$ into the measure $d\xi$ 
on the fiber $A^*_{s(g)}$.
\end{remark}

Since a symbol of order $-\infty$ defines a smoothing operator which can be viewed as an element in 
$\calA$, we can take its trace as defined in \eqref{trace}. The following is a straightforward 
computation with the kernel as given above:
\begin{proposition}
\label{prop:trace-dual}
For $a\in\symb^{-\infty}\left(A^*\right)$:
\[
\tau_\Omega\left(\op(a)\right)=\int_{A^*} a\Omega.
\]
\end{proposition}
%\begin{corollary}
%If $\sfG$ is unimodular, the trace on $\Psi^{-\infty}(\sfG)$ associated to the choice of an invariant transversal measure $\Omega$ satisfies
%\[
%\tau_\Omega\left(A-\op(\sigma(A))\right)=0,\quad A\in\Psi^{\infty}(\sfG).
%\]
%\end{corollary}

\subsection{The universal enveloping algebra}The universal enveloping algebra of the Lie algebroid 
$A$ is a concrete model for  the subalgebra of invariant differential operators along the $t$-fibers 
inside $\Psi^\infty(\sfG)$, cf.\ \cite{nwx}. This algebra is generated by $\calC^\infty(M)$ considered as 
multiplication operators in $\Psi^0(\sfG)$ and $\Gamma^\infty(M,A)$ viewed as first order invariant 
differential operators without constant term in $\Psi^1(\sfG)$ using the canonical isomorphism 
$s^*A\cong T_t\sfG$. If we write these two identifications as maps 
$i_1:\calC^\infty(M)\to\Psi^\infty(\sfG)$ and $i_2:\Gamma^\infty(M,A)\to \Psi^\infty(\sfG)$, 
we have the properties
\[
i_1(f)i_2(X)=i_2(fX),\quad [i_2(X),i_1(f)]=i_1(\rho(X)f),
\]
for $f\in \calC^\infty(M)$ and $X\in \Gamma^\infty(M,A)$, together with the fact that $i_2$ is a morphism 
of Lie algebras. These properties ensure that $i_1$ and $i_2$ lift to an algebra morphism from the 
universal enveloping algebra $i:\mathcal{U}(A)\to\Psi^\infty(\sfG)$ whose image are exactly the 
invariant differential operators. 

The Poincar\'e--Birkhoff--Witt theorem for Lie algebroids (cf.\ \cite[Thm 3]{nwx}) states that the restriction of the symbol map to the image of $\calU(A)$ in $\Psi^\infty(\sfG)$ gives an isomorphism of filtered vector spaces
\[
\sigma_\nabla\circ i:\calU(A)\to\Gamma^\infty(M,\symm A ),
\]
where $\symm A$ stands for the symmetric algebra bundle of $A$. Note that there is a natural
embedding $\Gamma^\infty(M,\symm A )  \hookrightarrow \symb^\infty (A^*)$ which is unique\-ly determined by the 
universal property of $\symm A$ and the requirement that a section  $X \in \Gamma^\infty (M, A)$ is mapped 
to the symbol 
\[
  \hat{X} : A^* \rightarrow \C, \quad  \xi \mapsto \lb \xi , X (\pi (\xi)) \rb .
\]
Hereby, $\pi : A^* \rightarrow M$ denotes the canonical projection.

\subsection{Elliptic operators and the index class}
We now turn to elliptic operators and use the symbol calculus to define the index class.
For a vector bundle $E$ over $M$, we define $\calU(A;E):=\calU(A)\otimes\End(E)$. The embedding of the previous section identifies 
this algebra as the algebra of invariant differential operators on sections of the vector bundle $t^*E$ inside $\Psi^\infty(\sfG;E)$. 
With these coefficients, the principal symbol map  \eqref{principal-symbol} restricts for each $k\geq 0$ to
\[
  \sigma_k\circ i:\calU_k(A;E)\to\Gamma^\infty(M,\symm^k(A)\otimes\End(E))\subset\Gamma^\infty(A^*,\End(\pi^*E)),
\]
where $\symm^k(A)$ is the bundle of $k$-fold symmetric powers of $A$. With this we now make the usual definition:
\begin{definition}
\label{elliptic}
An element $D\in \calU_k(A;E)$ is called {\em elliptic} if 
\[
  \sigma_k(i(D))(\xi)\in \End(E)
\] 
is invertible for all $\xi\in A^*\backslash M$.
\end{definition}
Clearly, this condition is equivalent to requiring that the differential operator $i(D)_x$ on $\Gamma(t^{-1}(x),t^*E)$ is elliptic 
for each $x\in M$. By definition, an elliptic element determines a class
\[
  [\sigma(i(D))]\in K^1(A^*\backslash M)\stackrel{\partial}{\longrightarrow} K^0(A^*).
\]
\begin{remark}
On a given Lie algebroid $A\to M$ there are plenty of elliptic operators in the sense above. 
As in \cite{ln}, one can construct examples by mimicking the standard constructions of elliptic differential operators on manifold. For example, if we fix a metric on $A$, we can define a Hodge $*$-operator and consider the ``signature operator''
\[
d\pm *d*\in\calU\left(A;\textstyle\bigwedge^\bullet A^*\right),
\]
with $d$ the Chevalley--Eilenberg operator of Lie algebroid cohomology \eqref{lie-alg-diff}.
This operator is elliptic in the sense of Definition \ref{elliptic} exactly by the same argument that the usual signature operator on a manifold is elliptic. In the same spirit, the Dirac operators constructed in \cite{ln} are examples of elliptic operators on a Lie algebroid.
\end{remark}

\begin{proposition}
An elliptic element $D\in\calU(A;E)$ canonically defines a localized index class
\[
  [\Ind(D)]_\textup{loc}\in K_0^\textup{loc}(\calA).
\]
\end{proposition}
\begin{proof}
With the operator-symbol calculus at our disposal, the proof follows the standard pattern: using the symbol calculus, we find 
$e\in\symb^\infty(A^*)$ such that
\[
  e\sigma(D)-1\in \symb^{-\infty}(A^*)\quad\mbox{and}\quad  \sigma(D)e-1\in\symb^{-\infty}(A^*).
\]
For any open neighbourhood $U$ of $M$ in $\sfG$, choose a neighbourhood $U'$ with $(U'_\partial)_\partial \subset U$, together with a cut-off function $\phi$ with support in $U'$. The corresponding quantization $\Op(e)$ will satisfy
\[
\Op(e)D-1\in\Psi^{-\infty}(\sfG)\quad\mbox{and}\quad  D\Op(e)-1\in\Psi^{-\infty}(\sfG),
\]
and we see from \eqref{reduced-kernel} that their reduced kernels have support in $U'$. Write $E=\Op(e)$, and define $S_0:=I-ED$, and $S_1:=I-DE$, and
\[
L=\left(\begin{array}{cc}S_0& -E-S_0E\\ D&S_1\end{array}\right).
\]
Then the matrix $R$ defined by
\[
R=L\left(\begin{array}{cc}I&0\\ 0&0\end{array}\right)L^{-1}-\left(\begin{array}{cc}0&0\\ 0&I\end{array}\right)
\]
is a formal difference of projectors in $M_2(\Psi^{-\infty}(\sfG)$ which defines
an element in $K_0\left(\calA(U)\right)$. Second, for a refinement
$V\subset U$ we have two element $R_U\in K_0\left(\calA(U)\right) $ and $R_V\in K_0\left(\calA(V)\right)$ defined by using cut-off functions
$\phi_\U$ and $\phi_V$. But then family of projectors $R_t,~t\in [0,1]$ defined using the cut-off function $\chi_t=t\chi_U+(1-t)\chi_V$ gives the desired homotopy proving that
both projectors define the same element in $K_0\left(\calA(U)\right)$.
In total, this defines the element $[\Ind(D)]_\textup{loc}\in K_0^\textup{loc}(\calA)$. It is independent of any choices made.
\end{proof}
Under the forgetful map $K_0^\textup{loc}(\calA)\to K_0(\calA)$, the proof shows that the class 
$[\Ind(P)]_\textup{loc}$ is mapped to the standard index class $[\Ind(P)]\in K_0(\calA)$ in 
noncommutative geometry (cf.~\cite{connes,mw,rouse}).
But in this paper we need the full class: it is our objective to compute the pairing 
\[
\llangle\alpha,[\Ind(P)]_\textup{loc}\rrangle \in\C,
\]
with an even Lie algebroid cohomology class $\alpha\in H^\textup{ev}_\textup{Lie}(A;\C)$.

\subsection{Asymptotic symbols and deformation quantization}
\label{as-def}
The idea is to compute the pairing above using an algebraic index theorem for the Poisson manifold $A^*$. 
To make the connection to deformation quantization, we need to  introduce the asymptotic family version 
of that calculus. This is closely related to the so-called ``adiabatic-'', or ``tangent''  deformation 
of the Lie groupoid $\sfG$, cf.~\cite{connes,nwx}.

An asymptotic symbol of order $m$ is a family of symbols $\hbar\mapsto a \in \symb^m_\textup{cl} (A^*)$, 
depending on a variable $\hbar\in [0,\infty)$, which by definition has an asymptotic expansion as 
$\hbar\to 0$ of the form
\[
a\sim\sum_{k=0}^\infty \hbar^k a_{m-k},
\]
with $a_{m-k}\in\symb_\textup{cl}^{m-k} (A^*)$. We write $\asymb^m(A^*)$ for the space of asymptotic symbols of order $m$
on $A^*$, and $\jsymb^m(A^*) $ for the asymptotic symbols vanishing to all orders in $\hbar$. 
The quotient $\mbA^m:=\asymb^m(A^*)\slash\jsymb^m(A^*)$ can be identified with $\symb^m(A^*)[[\hbar]]$.

The operator product of pseudodifferential operators induces an asymptotically associative product on the 
space of asymptotic symbols by
\[
 a \circ b   :=
   \begin{cases}
      \sigma_\hbar \big( \Op_\hbar (a) \circ \Op_\hbar (b) \big), & \text{if $\hbar > 0$} ,\\
      a(-,-,0)b(-,-,0), & \text{if $\hbar =0$}.
      \end{cases}
\]
Let $\Op_\hbar = \Op \circ \iota_\hbar$ and $\sigma_\hbar = \iota_{\hbar^{-1}}\circ\sigma$,
where $\iota_\hbar : \symb^\infty (A^*) \rightarrow \symb^\infty (A^*)$
is the map which maps a symbol $a$ to the symbol
$(x,\xi) \mapsto a(x,\hbar \xi)$. By the same standard techniques of pseudodifferential calculus as in 
\cite{P1} one checks that $\circ$ has an asymptotic expansion of the form
\begin{equation}
\label{eq:prod-as}
  a \circ b \sim a b  + \sum_{k =1}^\infty c_k (a,b) \, \hbar^k ,
\end{equation}
where the $c_k$ are bidifferential operators on $A^*$ and
\begin{displaymath}
    c_1 (a,b) -c_1 (b,a) = -i \{a,b\} \quad
    \text{for all $a,b \in \symb^\infty (A^*)$},
\end{displaymath}
where $\{~,~\}$ is the canonical Lie--Poisson bracket on $A^*$, cf. \eqref{poisson-dual-la} below.
On the quotient $\mbA_{-\infty}$ this defines a star product $\star_{\rm an}$ which is a deformation quantization 
of the Poisson manifold $A^*$. We write $\mathscr{S}^\hbar_{A^*}$ for the resulting sheaf of noncommutative 
algebras on $A^*$.

The trace $\tau_\Omega$ in Proposition \ref{prop:trace-dual} on $\symb^{-\infty}(A^*)$ defines a 
$\C[\hbar^{-1},\hbar]]$-valued trace on $\mbA_{-\infty}\subset\mbA_{\infty}$ given by
\begin{equation}
\label{eq:trace-def}
 \tau_\Omega(a):=\frac{1}{\hbar^r}\int_{A^*}a\Omega.
\end{equation}

\section{An algebraic index theorem for $A^*$}
\label{sec:formality}

\subsection{Regular Poisson manifolds}
The algebraic index theorem for regular Poisson manifolds is proved in \cite{nt-hol}, as a special of a symplectic Lie algebroid. 
For our purposes, we need a local version proved along the lines of our previous work \cite{ppt}. 

Let $(P,\Pi)$ be a regular Poisson manifold of dimension $n$. This means that the image of the map $\Pi^\#:T^*P\to TP$ defined by 
contraction with the Poisson tensor defines a regular foliation $\calF\subset TP$ whose leaves are symplectic. We can introduce local 
Darboux coordinates $q_1,\ldots,q_k,p_1,\ldots,p_k,x_1,\ldots x_r,~ 2k+r=n$ such that the foliation is given by the equations 
$x_1=\ldots=x_r=0$, and the Poisson tensor is in normal form:
\begin{equation}
\label{darboux-poisson}
\Pi=\sum_{i=1}^k\frac{\partial}{\partial q_i}\wedge\frac{\partial}{\partial p_i}.
\end{equation}

\subsubsection*{The Fedosov construction}
In \cite{fedosov}, Fedosov proves that regular Poisson manifolds admit deformation quantizations 
by applying the construction of symplectic manifolds leafwise. For this we first form the bundle of  
formal Weyl algebras
\[
  \calW_\calF:=\widehat{\symm}(\calF),
\]
equipped with the fiberwise Moyal product $\star$ using the fact that $\calF$ is symplectic. Next, we fix 
a ``symplectic connection along the leaves'' $\nabla:\Gamma(P,\calF)\to\Gamma(P,\calF\otimes\Omega^1_\calF)$. 
This is just a Lie algebroid connection on $\calF$, viewed as a Lie algebroid. The connection has curvature 
$R\in\Omega^2_\calF(M,\mathfrak{sp}(\calF))$.
The Fedosov iteration procedure shows that there exists an element $A\in\Omega^1_\calF(M;\calW_\calF)$ such 
that the Weyl curvature of the leafwise connection $D=\tilde{\nabla}+[A\slash\hbar,-]$ on $\calW_\calF$ given 
by
\[
  \Omega:=\tilde{R}+\tilde{\nabla}A+\frac{1}{2\hbar}[A,A]\in \Omega^2_\calF(M,\calW_\calF),
\]
is in fact central, i.e., takes values in $\C[[\hbar]]$. This means that $D$ is a flat connection and 
defines a complex
\[
 0\longrightarrow \mathscr{A}^\hbar(P)\longrightarrow\Omega^0_\calF(P,\calW)\stackrel{D}{\longrightarrow}\ldots \stackrel{D}{\longrightarrow} \Omega^{2k}_\calF(P,\calW).
\]
Furthermore, $D$ is compatible with the $\star$-product, and we can identify $\calC^\infty(P,\C[[\hbar]])$ 
with the space of flat sections for $D$, which therefore automatically inherits a $\star$-product. Since 
the $\star$-product is local, we obtain a sheaf $\mathscr{A}^\hbar_P$ of noncommutative algebras over 
$\C[[\hbar]]$. Observe  that by construction, the center of $\mathscr{A}^\hbar(P)$ consists of all $\calF$-basic 
functions, i.e., $f\in \calC^\infty(P)\otimes\C[[\hbar]]$ satisfying
\[
  L_X(f)=0,\quad\mbox{for all}~X\in \calF.
\]

Let us briefly discuss the classification of such star-products. Recall that an equivalence between two 
deformation quantizations $\star$ and $\star'$ is given by a formal series of differential operators of 
the form
\[
  \calE={\rm id}+\sum_{k=1}^\infty \hbar^k\calE_k,
\] 
such that $\calE(a\star b)=\calE(a)\star'\calE(b)$.
\begin{proposition}[cf.\  \cite{nt-hol}]
\label{class-def}
Let $P$ be a regular Poisson manifold foliated by symplectic leaves $\calF$.
\begin{itemize}
\item[$i)$] 
  Let $\star$ be a  star-product with the property that for any open subset $U$ and an $\calF$-basic 
  function $f\in \calC^\infty(U)$,
\[
  f\star g=fg,\quad\mbox{for all}~ g\in \calC^\infty(U).
\]
Then there exists a Fedosov connection along the leaves of $\calF$, together with an equivalence $\calE$ between $\star$ and the induced Fedosov star product, whose coefficients $\calE_k$ are elements of $\calU(\calF)$.
\item[$ii)$] Star products with the property stated in $i)$ are classified by an element in
\[
\frac{1}{\hbar}\omega+\sum_{k=0}^\infty\hbar^k\Omega_k\in \frac{1}{\hbar}\omega+H^2_\textup{Lie}(\calF,\C[[\hbar]]).
\]
\end{itemize}
\end{proposition}
\begin{proof}
The condition stated in $i)$ is equivalent to saying that the bidifferential operators appearing in the star product, are along the leaves of $\calF$. With this, the classification is proved in \cite[Thm 4.1]{nt-hol}. 
\end{proof}
\subsubsection*{The algebraic index theorem}
For $\star$-products obtained by this Fedosov procedure, the methods of \cite{ppt} apply in a straightforward way to give an algebraic index theorem. This works, because for regular Poisson manifolds we can just work along the leaves, so the only major difference is that instead of the de Rham complex, we use the complex $(\Omega^\bullet_\calF,d_\calF)$ of leafwise differential forms. Below we give an outline of this derivation, the proofs of all statements are simple adaptions of the proofs of similar statements in \cite{ppt}.

Using the generating cyclic cocycle on the Weyl algebra, one constructs the ``cyclic trace density'' morphism
\[
\Psi^i_{2m}:\calC_{2m-i}(\calW)\to\Omega^i_\calF\otimes\C[\hbar^{-1},\hbar]],
\]
where $\calC_\bullet(\calW)$ is the sheafified version of the Hochschild chains on $\calW$, by exactly the same formula as in \cite[Def. 3.5]{ppt}.
This morphism of sheaves has the fundamental property
\begin{equation}
\label{fund-prop}
(-1)^id_\calF\circ\Psi^i_{2k-2m}=\Psi^{i+1}_{2m-2k}\circ b+\Psi^{i+1}_{2k-2m+2}\circ B.
\end{equation}
On the level of cyclic complexes, this defines a morphism of complexes of sheaves:
\[
\Psi:\left(\Tot_\bullet\left(\mathcal{BC}(\mathscr{A}^\hbar_P)\right),b+B\right)\to \left(\Tot_\bullet\left(\mathcal{B}\Omega_\calF\otimes\C[\hbar^{-1},\hbar]]\right),d_\calF\right).
\]
Let us construct another map between these complexes and state the algebraic index theorem. 
First, recall that the smooth version of the Hochschild--Kostant--Rosenberg theorem states that the canonical map
\[
f_0\otimes\ldots\otimes f_p\mapsto f_0df_1\wedge\ldots\wedge df_p
\]
defines a quasi-isomorphism $(\calC_\bullet(\calC^\infty_P),b)\to(\Omega^\bullet_P,0)$ of complexes of sheaves, which sends the $B$-operator to the de Rham differential. Therefore, using the canonical projection $\mathscr{A}^\hbar_P\to \calC^\infty_P$ by putting $\hbar$ equal to zero, applying the Hochschild--Kostant--Rosenberg map and restricting to the leaves $\calF\subset TP$, we obtain a morphism 
\[
\sigma: \left(\Tot_\bullet\left(\mathcal{BC}(\mathscr{A}^\hbar_P)\right),b+B\right)\to \left(\Tot_\bullet\left(\mathcal{B}\Omega_\calF\right),d_\calF\right).
\]
Finally, introduce the characteristic class
\[
\hat{A}(\calF):=\prod_{i=1}^k\frac{ x_i\slash 2}{\sinh(x_i\slash 2)}\in H^{\rm ev}_{\rm Lie}\left(\calF;\C\right),
\]
where the $x_i$ are the leafwise Chern classes with respect to a almost complex structure compatible with the symplectic form. Now we have all ingredients to state the algebraic index theorem for regular Poisson manifolds:
\begin{theorem}[{cf.\ \cite[Thm 6.1.]{nt-hol}}]
\label{alg-ind-reg-poisson}
In the derived category of sheaves, the following diagram commutes:
\[
  \xymatrix{\Tot_\bullet\left(\mathcal{BC}(\mathscr{A}^\hbar_P)\right)\ar[rr]^{\sigma}\ar[drr]_{\Psi}&&
  \Tot_\bullet\left(\mathcal{B}\Omega_\calF\right) \ar[d]^{\wedge\hat{A}(\calF)e^{-\Omega\slash 2\pi\sqrt{-1}\hbar}}
  \\
  &&
  \Tot_\bullet\left(\mathcal{B}\Omega_\calF\right)\otimes\C[\hbar^{-1}, \hbar]]}
\]
\end{theorem}

\subsection{The algebraic index theorem for $A^*$}\label{sec:alg-ind}
The Poisson manifold that we are interested in is the dual $A^*$ of the Lie algebroid $A$. 
Its Poisson structure is characterized by
\begin{equation}
\label{poisson-dual-la}
\{f,g\}=0,\quad \{f,\hat{X}\}=-\rho(X)\cdot f,\quad \{\hat{X},\hat{Y}\}=-\widehat{\{X,Y\}},
\end{equation}
with $f,g\in \calC^\infty(M)$, and $X,Y\in\Gamma^\infty (M,A)$, viewed as functions $\hat{X},\hat{Y}$  on 
$A^*$, linear on the fibers. 
This Poisson structure is not regular in general, unless the anchor is injective. 
However, $A^*$ is always the quotient of a regular Poisson manifold: again we look at the map 
$t:\sfG\to M$, and define 
\[
T^*_t\sfG:=\ker(dt)^*\subset T^*\sfG. 
\]
If we view $\sfG$ as a foliated manifold with leaves $\sfG_x,~x\in M$ given by the $t$-fibers, $T^*_t\sfG$ is foliated with leaves $T^*\sfG_x,~x\in M$, which are just the fibers of the map $\tilde{t}:T^*_t\sfG\to M$ which is just $t$ precomposed with the canonical projection $T^*_t\sfG\to\sfG$. As cotangent bundles, the leaves are canonically symplectic, and we see that $T^*_t\sfG$ is a regular Poisson manifold.

Next, observe that $\sfG$ acts on itself by right translation, $R_g(h)=hg$,  and this action lifts to an action on
$T^*_t\sfG$ with moment map $\tilde{t}$. The quotient of this action is identified with $A^*$ by the map
$R:T^*_t\sfG\to A^*$ which maps $\xi\in T^*_t\sfG|_g$ to
\[
R(\xi)=R^*_{g^{-1}}(\xi)\in T^*_t\sfG|_{s(g)}=A^*_{s(g)}.
\]  
The map $R$ is a left inverse to the canonical inclusion $i:A^*\hookrightarrow T^*_t\sfG$ given by the identification $T^*_t\sfG|_M\cong A^*$.

It is proved in \cite[Lemma 7]{nwx} that $R$ is a Poisson map, i.e.,
\[
R^*\{f,g\}_{A^*}=\{R^*f,R^*g\}_{T^*_t\sfG},\quad\mbox{for all}~f,g\in \calC^\infty(A^*).
\]

As in the previous section, we denote by $\calF\subset T(T^*_t\sfG)$ the symplectic foliation of the 
Poisson structure on $T^*_t\sfG$.
On the other hand, recall the pull-back construction $\pi^!A$ of a Lie algebroid $A\to M$ over a 
submersion $\pi:P\to M$ with connected fibers, cf.\ \cite{mack}. 
This is defined as the vector subbundle of $\pi^*A\oplus TP$ given by
\[
(\pi^!A)_p:=\{(\xi,X)\in A_{\pi(p)}\oplus T_pP \mid \rho(\xi)=d\pi(X)\}. 
\]
The Lie bracket and anchor are described in \cite{mack}. In our case, we use the canonical projection $\pi:A^*\to M$ to pull-back $A$:
\begin{lemma}
\label{id-lie}
There is a canonical isomorphism $\calF|_{A^*}\cong\pi^!A$.
\end{lemma}
\begin{proof}
Let $T_t\sfG$ be the subbundle of $T\sfG$ given by the kernel of $dt:T\sfG\to 
TM$.  $T_t\sfG$ is a vector bundle over $\sfG$, and the groupoid $\sfG$ acts 
naturally on $T_t\sfG$ by right multiplication. The action is free 
and proper with the quotient being $A=T_t\sfG|_M$. The differential of the 
natural projection map from $T^*_t\sfG$ to $\sfG$ maps $\calF$ to 
$T_t\sfG$. We have a natural commutative diagram
\[
\xymatrix{\calF\ar[r]^{\pi_{T_t\sfG}}\ar[d]_{\pi_{T_t^*\sfG}}&T_t\sfG\ar[d]^{\pi_\sfG}\\
T^*_t\sfG\ar[r]^{\pi_\sfG}&\sfG.}
\]
The crucial property of the diagram is that it is compatible with the 
right $\sfG$ action, and on the quotient gives a commutative diagram
\begin{equation}
\label{diagram}
\xymatrix{\calF|_{A^*}\ar[r]^{\pi_A}\ar[d]_{\pi}&A\ar[d]^{\pi_{A^*}}\\
A^*\ar[r]^{\pi}&M
}
\end{equation}
From the above diagram, one can easily obtain the desired isomorphism from 
$\calF|_{A^*}$ to $\pi^! A$, as well the fact that this isomorphism is compatible with the Lie algebroid structure on $\pi^!A$ described in  \cite{mack}. 
\end{proof}
In our constructions, we shall need the following structures on the Lie algebroid $\pi^!A$:
\begin{itemize}
\item[$i)$] {\em The symplectic structure.} 
The Lie algebroid $\pi^!A$ carries a canonical nondegenerate closed two-form $\Theta\in\Omega^2_{\pi^!A}$. 
This can be seen as follows: consider the diagram \eqref{diagram} into which $\pi^!A$ fits. 
With this the symplectic structure $\Theta\in\bigwedge^2(\pi^!A)^*$ is given by 
\[
\Theta(X,Y)=\left<\pi_A(X),\pi_{A^*}(Y)\right>-\left<\pi_A(Y),\pi_{A^*}(X)\right>,
\]
where $\left<~,~\right>$ denotes the dual pairing between $A$ and $A^*$.
\item[$ii)$] {\em The Euler vector field.} As a vector bundle, $A^*$ carries a canonical $\R^*$-action, 
generated by the Euler vector field $\xi$. In a local frame $e_i\in\Gamma^\infty(M,A^*), i=1,\ldots, r$, we have 
\[
  \xi=\sum_{i=1}^re_i\frac{\partial}{\partial e_i}.
\]
Clearly, $d\pi(\xi)=0$, so we can consider $\xi$ as a section of $\pi^!A$. 

\end{itemize}

Recall the notion of a Lie algebroid connection on $\pi^!A$: this is a map 
\[
  \nabla: \Gamma^\infty(M,\pi^!A)\to\Gamma\left(M,\pi^!A\otimes(\pi^!A)^*\right),
\]
satisfying
\[
  \nabla_X(fY)=f\nabla_X(Y)+(\rho(X)\cdot f)Y,
\]
for all $f\in \calC^\infty(A^*)$ and $X,Y\in\Gamma^\infty(M,\pi^!A)$. A partition of unity argument shows that such connections always exist. With the structures on $\pi^!A$ described above, we shall say that a Lie algebroid connection is {\em symplectic} if
\[
  \nabla_X\Theta(Y,Z)=\rho(X)\cdot \Theta(Y,Z)-\Theta(\nabla_XY,Z)-\Theta(Y,\nabla_XZ)=0,
\]
for all $X,Y,Z\in\Gamma^\infty(A^*,\pi^!A)$. Similarly, it is said to be {\em homogeneous} if for all 
$X,Y\in\Gamma^\infty(A^*,\pi^!A)$
\[
  \rho(\xi)\cdot\nabla_X(Y)-\nabla_{[\xi,X]}(Y)-\nabla_X([\xi,Y])=0.
\]
\begin{proposition}\label{lem:conn}
Choosing an $\pi^!A$-connection $\nabla$ on $\pi^!A$ is equivalent to choosing a $\sfG$-invariant connection 
$\tilde{\nabla}$ on $T^*_t\sfG$ along the leaves $\calF$. 
\begin{itemize}
\item[$i)$] 
  $\tilde{\nabla}$ is Poisson if and only if $\nabla$ is symplectic,
\item[$ii)$] 
  $\tilde{\nabla}$ is homogeneous if and only if $\nabla$ is homogeneous.
\end{itemize}
\end{proposition}
\begin{proof}
By Lemma \ref{id-lie}, the pull-back $R^*\nabla$ of any $\pi^!A$-connection on $\pi^!A$ is a $\sfG$-invariant 
connection on $T^*_t\sfG$ along the leaves of $\calF$. Conversely, the restriction of any connection on 
$T^*_t\sfG$ along $\calF$ is a $\pi^!A$-connection on $\pi^!A$.
One easily checks that under the isomorphism of Lemma \ref{id-lie}, the symplectic structure and the Euler 
vector field are just the natural restrictions of the symplectic structure and the Euler vector field on $\calF$. 
With this observation, the proof is straightforward.
\end{proof}
Given a symplectic connection on $\pi^!A$, we can apply the Fedosov method using the connection $\tilde{\nabla}$ 
to obtain a $\sfG$-invariant star product on $T^*_t\sfG$, and we write $\mathscr{A}^\hbar_{T^*_t\sfG}$ for the 
resulting sheaf of noncommutative algebras.  Because of the $\sfG$-action on $\mathscr{A}^\hbar_{T^*_t\sfG}$, 
we can define
\[
  f\star_{A^*}g:=\left((R^*f)\star_{T^*_t\sfG} (R^*f)\right)\big|_{A^*}
\]
to define a star-product on $A^*$. Indeed, since $R$ is classically a Poisson map and everything is 
$\sfG$-equivariant, we easily see that this defines a deformation quantization of $A^*$ equipped with the 
Poisson bracket \eqref{poisson-dual-la}. We write $\mathscr{A}^\hbar_{A^*}$ for the resulting sheaf on $A^*$.

\begin{remark}
  Let us emphasize at this point that in the construction above, $\pi^!A$ served only as an auxiliary tool
  to construct a quantization over $A^*$, and that we neither construct nor use a quantization over 
  the symplectic Lie algebroid $\pi^!A$.
\end{remark}

As previously, we denote by $\calF\subset T(T^*_t\sfG)$ the symplectic foliation of the Poisson structure on 
$T^*_t\sfG$, and write $\Omega^\bullet_\calF(T^*_t\sfG)$ for the complex of leafwise differential forms equipped 
with the de Rham differential $d_\calF$. The subcomplex $(\Omega^\bullet_\calF(T^*_t\sfG),d_\calF)^\sfG$ can be 
interpreted as the Lie algebroid cohomology of a Lie algebroid over $A^*$, given by the restriction of 
$\calF$ to $A^*$.

With this, the algebraic index theorem is given by:
\begin{theorem}\label{thm:psi}
In the derived category of sheaves, the following diagram commutes:
\[
  \xymatrix{\Tot_\bullet\left(\mathcal{BC}(\mathscr{A}^\hbar_{A^*})\right)\ar[rr]^{\sigma}\ar[drr]_{\Psi}&&
  \Tot_\bullet(\mathcal{B}\Omega_{\pi^!A})\ar[d]^{\wedge\hat{A}(\pi^!A)e^{-\Omega\slash 2\pi\sqrt{-1}\hbar}}
\\
& &
\Tot_\bullet(\mathcal{B}\Omega_{\pi^!A})\otimes\C[\hbar^{-1},\hbar]]}
\]
\end{theorem}
\begin{proof}
The theorem is a direct consequence of Theorem \ref{alg-ind-reg-poisson}, applied to the regular Poisson manifold 
$T^*_t\sfG$: the groupoid $\sfG$ acts on the sheaves $\mathscr{A}^\hbar_{T^*_t\sfG}$ and $\Omega^\bullet_\calF$ and its 
invariant section are resp. $\mathscr{A}^\hbar_{A^*}$ and the Lie algebroid cochain complex $\Omega^\bullet_{\pi^!A}$, 
because of Lemma \ref{id-lie}. The explicit formula for the morphism $\Psi$, cf. \cite[Def. 3.5.]{ppt} shows that 
it only depends on the choice of a Fedosov connection. Because we have chosen a $\sfG$-invariant connection, it 
follows that 
$\Psi: \Tot_\bullet(\mathcal{BC}(\mathscr{A}^\hbar_{T^*_t\sfG}))\to\Tot_\bullet( \mathcal{B}\Omega^\bullet_\calF)$ is a 
$\sfG$-equivariant morphism of chain complexes of  sheaves, and therefore restricts to invariants. 
With this observation, Theorem \ref{alg-ind-reg-poisson} gives the result.
\end{proof}
\begin{remark}{\ } 
\begin{itemize}
\item[$i)$] If we choose a homogeneous connection as in Proposition \ref{lem:conn}, we can, as in \cite{bnw}, make sure that the full Fedosov connection is homogeneous. It follows that the induced characteristic class $\omega\slash\hbar+\sum_{k=0}^\infty\hbar^k\Omega_k$ of Proposition \ref{class-def} $ii)$ satisfies $L_\xi\Omega_k=\Omega_k$ for all $k$. Since also $L_\xi\omega=\omega$, it follows that the characteristic class is zero in $H^2_{\rm Lie}(\pi^!A;\C)$, and does not contribute in the theorem above. This is important in the following, because the $\star$-product induced by the pseudodifferential calculus is homogeneous. 
\item[$ii)$]
  The algebraic index theorem for Poisson manifolds of \cite{tt,cfw} gives a somewhat different result when 
  applied to the dual of a Lie algebroid. It involves the ordinary de Rham complex instead of the Lie algebroid 
  cochain complex. Therefore, the $\hat{A}$-genus appearing in that theorem is the ordinary genus of the underlying 
  manifold. Here the $\hat{A}(\pi^!A)$ is quite different, for example it can be nontrivial even in the case 
  of the dual of a Lie algebra.
  \end{itemize}
\end{remark}

Finally, let us consider the cohomology of $\pi^!A$ and its structure.
Since $\pi:A^*\to M$ is a vector bundle, its fibers are contractible and  by \cite[Thm 2]{crainic} we immediately 
conclude:
\begin{corollary}\label{cor:alg-coh}
  The map $\pi^*:H^\bullet_\textup{Lie}(A)\to H^\bullet_\textup{Lie} (\pi^!A)$ is an isomorphism preserving the product 
  structure.
\end{corollary}
Next we need to discuss duality.
\begin{lemma}
\label{lema:volume}
The tensor product
\[
  \Omega_{\pi^!A}:=\Theta^r\otimes\Omega\in
  {\bigwedge}^\textup{\!top} \,(\pi^!A) \otimes{\bigwedge}^\textup{\!top}A\otimes{\bigwedge}^\textup{\!top}T^*M
\]
defines a $\pi^!A$-invariant volume form in $L_{\pi^!A}$ on $A^*$.
\end{lemma}

With this volume form, we can construct the following pairing
\[
  H^k_\textup{Lie}(\pi^!A)\times H^{2r-k}_{\rm Lie, c}(\pi^!A)\to\C,
\]
the $c$ stands for compact support, given by
\[
\left<\alpha,\beta\right>_{\pi^!A}:=\int_{A^*}\left<\alpha\wedge\beta,\Omega_{\pi^!A}\right>.
\]
Combining the pairing with the cyclic trace density, we obtain a morphism of cochain complexes 
$\Phi: \Omega^\bullet_\textup{Lie}(A;\C)\to \Tot^\bullet(\mathcal{BC}(\mathscr{A}^\hbar_{A^*,c}(A^*)))$ (the $c$ stands for compact support) given by
\begin{equation}
\label{ind-map}
  \Phi(\alpha)(a):=\left<\pi^*\alpha,\Psi(a)\right>_{\pi^!A}.
\end{equation}
The algebraic index theorem immediately shows:
\begin{corollary}\label{cor:phi}
For $\alpha\in H^\bullet_\textup{Lie}(A;\C)$, the following equality holds true in $HC^\bullet(\mathscr{A}^\hbar_{A^*,c}(A^*))$:
\[
  \Phi(\alpha)(a)=\int_{A^*}\left<\pi^*\alpha\wedge\hat{A}(\pi^!A)\wedge\sigma(a),\Omega_{\pi^!A}\right>.
\]
\end{corollary}

\section{The localized index theorem}
\label{sec:local-index}

In this section we prove the main theorem of this paper, which computes the localized index pairing of an 
elliptic operator with a Lie algebroid cocycle:
\begin{theorem}
\label{main-theorem}
Let $A\to M$ be a unimodular integrable Lie algebroid with an invariant volume form $\Omega$.
For any elliptic $D\in\calU(A;E)$ and $\alpha\in H^{2k}_\textup{Lie}(A;\C)$ the following identity holds true;
\[
  \llangle[\Ind(D)]_\textup{loc},\alpha\rrangle =\frac{1}{(2\pi\sqrt{-1})^k}\int_{A^*}
  \left<\pi^*\alpha\wedge\hat{A}(\pi^!A)\wedge {\rm ch}(\sigma(D)), \Omega_{\pi^!A}\right>.
\]
\end{theorem}
\subsection{Analytic vs. algebraic quantization}
Recall that a longitudinal pseudodifferential operator is actually a family of invariant pseudodifferential 
operators on the $t$-fibers of a Lie groupoid. As such, instead of taking the symbol on $A^*$, we can also consider the invariant family of symbols on $T^*_t\sfG$. For symbols of order $-\infty$, this is compatible with the correspondence between the invariant family of kernels and the reduced kernel:
\[
\xymatrix{
\symb^{-\infty}(A^*)\ar[r]^{R^*}\ar[d]_{k_{\Op}}&\symb^{-\infty}(T^*_t\sfG)^\textup{inv}\ar[d]^{\tilde{k}_{\Op}}\\
\calA\ar[r]_{k\mapsto\tilde{k}}&\calC^\infty_\textup{c}(\sfG^{\times_t 2}).}
\]
By the same arguments as in Section \ref{as-def}, the asymptotic symbol calculus on $T^*_t\sfG$ leads to a deformation quantization of the invariant functions on $T^*_t\sfG$. This means that the coefficients $c_k$ in product $\star_\textup{an}$ are invariant bidifferential operators, and clearly the star product extends to $\calC^\infty(T^*_t\sfG)$. We write $\mathscr{S}^\hbar_{T^*_t\sfG}$ for the resulting sheaf of noncommutative algebras. 

In later computations we shall need the following fact:
\begin{lemma}
\label{module-op}
For $a\in \mathbb{A}^\infty(T^*_t\sfG)$ and $f\in \calC^\infty(\sfG)$, we have
\[
\sigma_\hbar\left( f \circ \Op_\hbar(a)\right)\sim f\star_\textup{an} a.
\]
\end{lemma}
\begin{proof}
From the formula for $\Op(a)$, it is clear that $f\circ \Op(a)=\Op(f a)$. But since the star product $\star_\textup{an}$ is normal ordered, it follows that $fa=f\star_\textup{an} a$, for $f\in \calC^\infty(\sfG)$.
\end{proof}
Next, we need to compare this star-product with the ones coming from Fedosov's quantization.
\begin{proposition}
On $T^*_t\sfG$, there exists a $\sfG$-invariant Poisson connection $\nabla$ together with a $\sfG$-equivariant equivalence of $\star$-products 
\[
\calE_{T^*_t\sfG}:\mathscr{A}^\hbar_{T^*_t\sfG}\to \mathscr{S}^\hbar_{T^*_t\sfG}.
\]
\end{proposition}
\begin{proof}
The star product of the symbol calculus clearly has the property that for any function $f$ that is constant along the fibers $f\star_{\rm an} g=fg$. Therefore $\star_{\rm an}$ satisfies the condition of Proposition \ref{class-def}, so there is an equivalence to some Fedosov-type star product. However, it is well-known that the characteristic class of the deformation quantization given by the symbol calculus equals zero, so by Proposition \ref{class-def} $ii)$ we can choose a $\sfG$-invariant Fedosov quantization together with an equivalence. This proves the statement.
\end{proof}

Restricting to $\sfG$-invariants, this defines an equivalence $\calE_{A^*}:\mathscr{A}^\hbar_{A^*}\to \mathscr{S}^\hbar_{A^*}$ making the following diagram commute:
\[
\xymatrix{
\mathscr{A}^\hbar_{A^*}\ar[r]^{\calE_{A^*}}\ar[d]_{R^*}&\mathscr{S}^\hbar_{A^*}\ar[d]^{R^*}\\
\mathscr{A}^\hbar_{T^*_t\sfG}\ar[r]^{\calE_{T^*_t\sfG}}&\mathscr{S}^\hbar_{T^*_t\sfG}}
\]
Finally, we can impose compatibility with the Euler vector field: the star product $\star_{\rm an}$ is known to be homogeneous, meaning that 
\[
\Xi:=\hbar\frac{\partial}{\partial\hbar}+L_\xi
\]
defines a derivation of $(\mathscr{S}^\hbar_{A^*},\star_{\rm an})$. As a consequence, $\star_{\rm an}$ restricts to 
the functions in $\calC^\infty(A^*)$ that are polynomial on the fiber. This subalgebra is nothing but the 
universal enveloping algebra $\calU(A)$, a quantization of $\Gamma^\infty(M,\symm (A))$ via the 
Poincar\'e--Birkhoff--Witt map.

Using Proposition \ref{lem:conn} we can obtain homogeneous Poisson connections on $T^*_t\sfG$ from homogeneous Lie algebroid connections. As shown in \cite[\S 3]{bnw}, one can choose the terms in the Fedosov iteration in such a way that the resulting Fedosov star product is homogeneous in the sense above. The resulting equivalence then commutes with the homogeneity operator:
\begin{equation}
\label{com-hom}
[\calE,\Xi]=0.
\end{equation}
\subsection{A comparison of traces}
Assume $\sfG$ is unimodular, and choose an invariant transverse measure $\Omega$. Restricted to the Hochschild subcomplex, the map $\Psi$ defines in degree $2r$ a morphism of sheaves
\[
\Psi^{2r}_{2r}:\calC_{0}(\mathscr{A}^\hbar_{A^*})\to \Omega^{2r}_{\pi^!A}\otimes\C[\hbar^{-1},\hbar]],
\]
satisfying
\[
\Psi^{2r}_{2r}(a\star b-b\star a)= d\mbox{-exact},
\]
by Eq. \eqref{fund-prop}.
It follows that the formula
\[
\tr_\Omega(a):=\int_{A^*}\Psi^{2r}_{2r}(a)\Omega_{\pi^!A}
\]
defines a trace on $\mathscr{A}^\hbar_{A^*}$.
\begin{proposition}
\label{comp-trace}
We have the following equality of traces:
\[
\tau_\Omega\circ \calE_{A^*}=\tau_\Omega=\tr_\Omega.
\]
\end{proposition}
\begin{proof}
We first prove $\tau_\Omega \circ \calE_{A^*}=\tau_\Omega$. This is analogous to \cite[Thm. 5.4]{bnpw}. For $x\in \sfG_0$, over the $t$-fiber of the projection $T^*_t\sfG\to \sfG_0$, the integration with respect to the Liouville volume form is a trace on both $\mathscr{A}^\hbar_{T^*_t\sfG}|_{t^{-1}(x)}$ and $\mathscr{S}^\hbar_{T^*_t\sfG}|_{t^{-1}(x)}$. This implies that for $a\in \Gamma(\mathscr{S}^\hbar_{A^*})$, $\calE_{T^*_t\sfG}(R^*(a))|_{t^{-1}(x)}-R^*(a)|_{t^{-1}(x)}$ is a coboundary $d\alpha(a)$ on $t^{-1}(x)$. As both $\calE_{T^*_t\sfG}(R^*(a))$ and $R^*(a)$ are $\sfG$-invariant, the form $\alpha(a)$ can be made invariant under the $\sfG$ action on $T^*_t\sfG$, and therefore is the pullback of a section of $\wedge^{2r-1} (\pi^!A)^*$.  Accordingly, $\calE_{A^*}(a)-a$ is an exact form in the Lie algebroid cochain complex of $\pi^!(A)$. As the pairing $\langle d_{\textup{Lie}}(\alpha(a)), \tilde{\Omega}\rangle=0$, we conclude that 
\[
\tau(\calE_A^*(a))=\langle \calE_A^*(a), \Omega_{\pi^!A}\rangle=\langle \Psi_{2r}(a), \Omega_{\pi^!A}\rangle=\tau(a). 
\] 
Similar arguments as above also prove $\tr_\Omega=\tau_\Omega$. Over each fiber $t^{-1}(x)$, it is proved in \cite[Prop. 6.8]{ppt} that $\Psi_{2r}(R^*(a))-a$ is an exact form $d\alpha(a)$. As above
it follows that 
\[
\tr_\Omega(a)=\langle \Psi_{2r}(a), \Omega_{\pi^!A}\rangle=\langle a, \Omega_{\pi^!A}\rangle=\tau_\Omega(a). 
\] 
This completes the proof.
\end{proof}

\subsection{Homogeneous cochains}
\label{homogeneous}
Recall that the cohomology of groups can be computed from both the homogeneous or inhomogeneous cochain complex. The groupoid version of the inhomogeneous complex was considered in Section\ref{diff-coh}. Here we describe the corresponding homogeneous version. We put
\[
  \sfG^{\times_{\!t}k}:=\{(g_1,\ldots,g_k)\in\sfG^k \mid t(g_1)=\ldots =t(g_k)\}.
\]
This space carries an action of $\sfG$ by 
%$h\cdot(g_1,\ldots,g_k)=(hg_1,\ldots,hg_k)$, 
$(g_1,\ldots,g_k) \cdot h=(g_1 h ,\ldots,g_k h)$, 
and we define 
$D^k_\textup{diff}(\sfG):=\calC^\infty_\textup{inv}(\sfG^{\times_{\!t} (k+1)})$. There is an isomorphism 
$C^k_\textup{diff}(\sfG)\cong D^k_\textup{diff}(\sfG)$ which sends a cochain 
$\varphi\in D^k_\textup{diff}(\sfG)$ to $\tilde{\varphi}\in C^k_\textup{diff}(\sfG)$ defined by
\[
\begin{split}
  \tilde{\varphi}(g_1,\ldots,g_k):=  \varphi \big(g_1\cdots g_k, g_2\cdots g_k,\ldots & , g_k, 1_{t(g_k)}\big),
\end{split}
\]
where  $(g_1,\ldots,g_k)\in\sfG_k$.
Conversely, an element $\psi\in C^k_\textup{diff}(\sfG)$ is mapped to $\bar{\psi}\in D^k_\textup{diff}(\sfG)$ given by
\[
\bar{\psi}(g_0,\ldots,g_k):=\psi(g_0 g_1^{-1} ,\ldots,g_{k-1} g_k^{-1} ),\quad (g_0,\ldots,g_k)\in\sfG^{\times_{\! t}(k+1)}.
\]
Under this isomorphism, the cosimplicial structure on $C^\bullet_\textup{diff}(\sfG)$ is transferred to $D^\bullet_\textup{diff}(\sfG)$ by
\begin{equation}
\label{Eq:cosimp}
  d_i\psi(g_0,\ldots,g_k):=\psi(g_0,\ldots,\hat{g}_i,\ldots,g_k),\quad\mbox{for}~0\leq i\leq k.
\end{equation}
In the same way, the cyclic structure on $D^\bullet_\textup{diff}(\sfG)$ is given by
\[
\tau_k\psi(g_0,\ldots,g_k):= \psi(g_k,g_0,\ldots,g_{k-1}).
\]
Taking the alternating sum $d:=\sum_i(-1)d_i$, we therefore obtain a complex $(D^\bullet_\textup{diff}(\sfG),d)$ 
computing differentiable groupoid cohomology. Likewise, the unit space $M$ is diagonally embedded in 
$\sfG^{\times_{\!t}k}$ and localizing the cochains in $C^\bullet_\textup{diff}(\sfG)$ to $M$ corresponds to localization in $ D^\bullet_\textup{diff}(\sfG)$ to $M$.  However, one concludes from  Eq.~\eqref{Eq:cosimp} that, after localization,
cochains in $\hat{D}^\bullet_\textup{diff}(\sfG)$ can be interpreted as families of smooth Alexander--Spanier cochains on 
the $t$-fibers of $t:\sfG\to M$ which are $\sfG$-invariant (see \cite{conmos} for details on Alexander--Spanier cohomology). For smooth manifolds, Alexander--Spanier cohomology is canonically isomorphic to the de Rham cohomology, and since the $\sfG$-action on itself is free and proper, the resulting complex $(\hat{D}^\bullet_\textup{diff}(\sfG),\hat{d})$ is quasi-isomorphic to $(\Omega^\bullet_t(\sfG),d_\textup{dR})^{\sfG}$, whose cohomology equals the Lie algebroid cohomology $H^\bullet_\textup{Lie}(A;\C)$. This gives an alternative proof of Prop \ref{loc-lie}.
\begin{proposition}
\label{pairing-hom}
For homogeneous cochains $\varphi\in D^k_\textup{diff}(\sfG)$, the characteristic map \eqref{eq:chi-uni} is given by
\[
\begin{split}
\chi(\tilde{\varphi})(a)&:=\left<\varphi,a_0\otimes\ldots\otimes a_k\right>\\
&=\int_M\int_{\sfG^k_x} \varphi(x,g_1,\ldots,g_k) \tilde{a}_0(x,g_1)\ldots \tilde{a}_k(g_k,x)dg_1\ldots dg_kdx.
\end{split}
\]
\end{proposition}
\begin{proof}
This is a straightforward computation.
\end{proof}

\subsection{Another characteristic map}\label{sec:another-ch}
In equation \eqref{ind-map}, we have defined a map $\Phi:H^\bullet_\textup{Lie}(A;\C)\to HC^\bullet(\mathscr{A}_{A^*}^\hbar)$. 
Here we shall give another construction of this map on the level of cochains. For this we use the cochain complex $\hat{D}^\bullet_\textup{diff}(\sfG)$ described above, as $\sfG$-invariant smooth Alexander--Spanier cochains on the $t$-fibers of the target map. To simplify the notation, we write $\psi(g_0,\ldots,g_k)=\psi_0(g_0)\otimes\ldots\otimes\psi(g_k)$ for a differentiable cochain in $D^k_\textup{diff}(\sfG)$. For the purpose of cyclic cohomology, we will work with the space of cyclic Alexander--Spanier cochains $D^k_{\lambda, \textup{diff}}(\sfG)$. Recall that this space consists of all
elements of $D^k_\textup{diff}(\sfG)$ which are invariant under the cyclic operator (cf.~\cite{ppt}  and \cite{kp}). 
In the limit that the cochains localize to $M$, this notation can be made precise, provided one reads the tensor products as being over $\calC^\infty(M)$. 
We then proceed as in \cite[Sec.~6]{ppt}: for $\psi\in D^k_{\lambda, \rm diff}(\sfG)$,  and 
$a=a_0\otimes\ldots a_k\in \calC_k(\calA^\hbar_{A^*})$, we define 
\begin{equation}\label{eq:frakX}
\begin{split}
  \mathfrak{X}&\, (\psi)(a) :=\\
  & \Psi^{2r}_{2r} \left. \left(R^*a_0\star \pi^*\psi_0\star\ldots
  \star R^*a_k\star\pi^*\psi_k \right)\right|_{A^*}\in \Omega^{2r}_{\pi^!A}\otimes\C[\hbar^{-1},\hbar]].
\end{split}
\end{equation}
\begin{proposition}
For $\psi\in D^{k-1}_{\lambda, \rm diff}(\sfG)$ and $a\in \mathcal{C}_k(\mathscr{A}^\hbar_{A^*})$, the quantities
\[
\begin{split}
  \mathfrak{X}(d\psi)(a)-\mathfrak{X}(\psi)(ba),\quad\mathfrak{X}(\psi)(Ba)
\end{split}
\]
are exact in $\Omega_{\pi^!A}^{2r}$. 
\end{proposition}
\begin{proof}  This is essentially a generalization of \cite[Sec. 6, Step I.]{ppt}. The following identity of $\Psi$ is shown in \cite[Prop. 4.2]{ffs} and \cite[Prop. 3.6]{ppt}, as $\Psi^{2r}_{2r+2}=0$:
\begin{equation}\label{eq:psi}
  \Psi_{2r}^{2r}\big( [a, b ]_\star \big)=\Psi^{2r}_{2r}(b (a\otimes b))=-d\Psi^{2r-1}_{2r}(a\otimes b). 
\end{equation}

Now it is straightforward to check that the difference $\mathfrak{X}(d\psi)(a)-\mathfrak{X}(\psi)(ba)$ is equal to 
\[
  (-1)^{k}\Psi^{2r}_{2r} \big( [a_0\star \psi_0\star\cdots\star a_{k-1}\star \psi_{k-1}, a_k ]_\star \big).
\]
By Eq. (\ref{eq:psi}), we conclude that the  above expression is equal to 
\[
  d\big((-1)^{k+1}\Psi^{2r-1}_{2r}(a_0\star \psi_0\star\cdots \star a_{k-1}\star \psi_{k-1}\otimes a_k)\big). 
\]
The similar computation proves that $\mathfrak{X}(\psi)(Ba)$ is exact with the assumption that $\psi$ is 
cyclic. We skip the tedious details. 
\end{proof}

\begin{corollary}
The map
\[
\overline{\mathfrak{X}}(\psi):=\int_{A^*}\left<\mathfrak{X}(\psi),\Omega_{\pi^!A}\right>
\]
defines a morphism of complexes
\[
\overline{\mathfrak{X}}:\left( D^\bullet_{\lambda, \rm diff}(\sfG),d\right)\to \left(\Tot^\bullet(\mathcal{BC}(\mathscr{A}^\hbar_{A^*})),b+B\right).
\]
\end{corollary}
\noindent We now come to the actual comparison of $\mathfrak{X}$ with $\Psi$.
\begin{proposition}\label{prop:frakX}
For $\Psi \in D^k_\textup{diff} (\sfG)$, we have the following equality in $H^\bullet_\textup{Lie}(\pi^!A;\C)$:
\[
\left[\mathfrak{X}(\psi)(a)\right]=\left[\Psi(a)\right]\cup \left[E(\psi)\right].
\]
\end{proposition}
\begin{proof} For $x\in M$, we look at the fiber $t^{-1}(x)$ in $T^*_t\sfG$, since $\Psi$ can be viewed as a fiberwise Alexander--Spanier cochain.
It has been proved in \cite[Prop. 6.8]{ppt} that  on each fiber $t^{-1}(x)$, 
\[
\int_{t^{-1}(x)} \mathfrak{X}(\psi)(a)=\int_{t^{-1}(x)}\Psi(a)\wedge E(\psi). 
\]
The above equation implies that $\mathfrak{X}(\psi)(a)-\Psi(a)\wedge E(\psi)$ is an exact form on $t^{-1}(x)$.  Define $\alpha_x(\psi, a)\in \Omega^{2r-1}_\calF( t^{-1} (x) )$ by
\[
d\alpha_x(\psi, a)=\mathfrak{X}(\psi)(a)-\Psi(a)\wedge E(\psi)\qquad x\in M.
\]
Notice that the $\sfG$ action on $T^*_t\sfG$ is proper and free. As both $\mathfrak{X}(\psi)(a)$ and $\Psi(a)\wedge E(\psi)$ are $\sfG$-invariant, $\alpha (\psi, a)\in \Omega^{2r-1}_\calF(T^*_t\sfG)$ can also be chosen to be $\sfG$-invariant using the trick of averaging with respect to a cut-off function on $T^*_t\sfG\rtimes \sfG$.  Therefore $\alpha(\psi, a)$ is in $\Omega^{2r-1}_\textup{Lie}(\pi^! A)$ satisfying
\[
d\alpha (\psi, a)=\mathfrak{X}(\psi)(a)-\Psi(a)\wedge E(\psi). 
\]
This gives the desired equality in the Lie algebroid cohomology of $\pi^!A$.
%Let us choose local Darboux coordinates in an open neighbourhood of $A^*$ in $T^*_t\sfG$ as follows: since $t:\sfG\to M$ is a submersion, we can choose local coordinates $(x,q)$ in a neighbourhood $U$ of $\sfG$ with $U\cap M\not=\emptyset$, where $x=x_1,\ldots,x_n$ are local coordinates on $M$ and the map $t$ is given by projection onto $x$. This induces local coordinates $(x,p,q)$ on the open neighbourhood $\pi^{-1}(U)$ of $T^*_t\sfG$ in which the Poisson tensor is given by \eqref{darboux-poisson}. By \cite[??]{ffs}, the map $\Psi_{2r}$ reduces in such coordinates to 
%\[
%\Psi_{2r}(a)=\calE(a)\Theta^r,
%\] 
%where $\calE$ is an equivalence with the Weyl algebra, which exists in such a chart. 
\end{proof}
\subsection{Proof of the main theorem}\label{sec:proof}
We now have all the ingredients to finish the proof of Theorem \ref{main-theorem}. As before, we denote by 
$[\alpha]\in H^\textup{ev}_\textup{Lie}(A;\C)$ a Lie algebroid cohomology class represented by a groupoid cochain 
$\alpha$ localized in a neighbourhood $U_k$ of the unit space, and by
$[\Ind(D)]_\textup{loc}\in K^\textup{loc}_0(\calA)$ the localized $K$-theory class of an elliptic 
pseudodifferential operator $D$ with parametrix defined by a kernel $k_D \in \calA$ with support in $U_1$. 
We write the kernel $k_D$ as $\Op(a)$ for some symbol $a\in\symb^{-\infty}(A^*)$. Since the localized 
$K$-theory class of $\Op_\hbar(a)$ does not depend on $\hbar$, we compute the pairing in the limit as 
$\hbar$ approaches $0$:
\[
\begin{split}
\llangle {[\Ind(P)]_\textup{loc}} & ,\varphi\rrangle  =\langle\chi(\varphi),k_P\rangle\\
=&\langle\varphi_0\otimes\ldots\otimes\varphi_{2k},k_P\otimes\ldots\otimes k_P\rangle\\
=&\lim_{\hbar\downarrow 0}\left<\varphi_0\otimes\ldots\otimes\varphi_{2k},\Op_\hbar(a)\otimes\ldots\otimes \Op_\hbar(a)\right>\\
&\mbox{(Prop. \ref{pairing-hom} and Lemma \ref{module-op})}\\
=&\lim_{\hbar\downarrow 0}\tau_\Omega\left(\Op_\hbar\left(\varphi_0\star_{\rm an}a_0\right)\cdots \Op_\hbar\left(\varphi_{2k}\star_{\rm an}a_{2k}\right)\right)\quad\\
=&\lim_{\hbar\downarrow 0}\tau_\Omega\left(\Op_\hbar\left(\varphi_0\star_{\rm an}a_0\star_{\rm an}\ldots\star_{\rm an}\varphi_{2k}\star_{\rm an}a_{2k}\right)\right)\\
=&\lim_{\hbar\downarrow 0}\tr_\Omega\left(\varphi_0\star_{\rm an}a_0\star_{\rm an}\ldots\star_{\rm an}\varphi_{2k}\star_{\rm an}a_{2k}\right)\hspace{10mm}\mbox{(Prop.~\ref{comp-trace})}\\
=&\lim_{\hbar\downarrow 0}\left<\Psi\left(a_0\otimes\ldots\otimes a_{2k}\right)\cup E(\varphi),\tilde{\Omega}\right> \hspace{17.5mm} \mbox{(Prop.~\ref{prop:frakX})}\\
=&\lim_{\hbar\downarrow 0}\int_{A^*}\left<\pi^*\alpha \wedge \hat{A}(\pi^!A)\wedge \sigma(\Ch(a)), \Omega_{\pi^!A} \right>.\quad \mbox{(Cor.~\ref{cor:phi})}\\
&\mbox{(as the integral is independent of $\hbar$)}\\
=&\int_{A^*}\left<\pi^*\alpha \wedge \hat{A}(\pi^!A)\wedge \sigma(\Ch(a)), \Omega_{\pi^!A} \right>.\qquad
\end{split}
\]
This completes the proof of Theorem \ref{main-theorem}.
\section{The  nonunimodular case}
\label{sec:nonunimodular}

Here, we explain how to remove the unimodular assumption in Theorem \ref{main-theorem}.  When $\sfG$ is nonunimodular, there exists no invariant
``transversal density'' $\Omega$, hence we do not have an obvious trace  on  the convolution algebra $\calA$ at our disposal.  However, we notice that the groupoid $\sfG$ and also the associated Lie algebroid $A$ naturally act on the bundle $L=\bigwedge^{\text{top}}T^*M\otimes \bigwedge^{\text{top}}A^*$. Therefore, we work with differentiable groupoid or Lie algebroid cohomology with coefficients in $L$ in the definition of the characteristic map $\chi$.  As many arguments in this section are direct generalizations of those in the unimodular case, we will skip the details of most proofs. 

Let $\sigma$ be the map from $\sfG_k$ to $M$ by mapping $(g_1, \cdots, g_k)$ to $t(g_k)$.  For $g\in \sfG$, denote by $\alpha_g: L|_{t(g)}\to L|_{s(g)}$ the action on $L$. The action $\alpha$ satisfies $\alpha_{gh}=\alpha_g\alpha_h$. Consider $C^k_\textup{diff}(\sfG;L):=\Gamma^\infty(\sfG_k;\sigma^*(L))$, the space of smooth sections of the line bundle $\sigma^*(L)$ on $\sfG_k$. On $C^k_\textup{diff}(\sfG;L)$, define the differential
\[
\begin{split}
d\varphi(g_1,&\ldots,g_{k+1}):=\alpha_{g_1^{-1}}\big(\varphi(g_2,\ldots,g_{k+1})\big)\\
&+\sum_{i=1}^k(-1)^k\varphi(g_1,\ldots,g_ig_{i+1},\ldots,g_{k+1})+(-1)^{k+1}\varphi(g_1,\ldots,g_k)\big.
\end{split}
\]
The differentiable cohomology $H^\bullet_{\textup{diff}}(\sfG;L)$ of $\sfG$ with coefficient in $L$ is defined to be the cohomology of the complex $(C^\bullet_\textup{diff} (\sfG;L), d)$. 

Analogous to the standard differentiable groupoid cohomology, we can consider the homogeneous cochain complex as  in Section \ref{homogeneous}.  Let $t$ be the natural map from $\sfG^{\times_{\!t}k}$ to $M$ by mapping $(g_1, \cdots, g_k)$ to $t(g_1)=\cdots =t(g_k)$.  Observe that $\sfG$ acts on both $\sfG^{\times_{\!t}k}$ and $L$. Define $D^k_{\textup{diff}}(\sfG; L)$ to be $\Gamma^\infty_\textup{inv}(\sfG^{\times_{\! t}(k+1)};t^*L)$ with the same differential as on $D^\bullet_{\textup{diff}}(\sfG)$,
\[
\begin{split}
d_i\psi(g_0, \cdots, g_k)&:=\psi(g_0, \cdots, \hat{g}_i, \cdots, g_k),\qquad 0\leq i\leq k\\
\tau_k\psi(g_0, \cdots, g_k)&:=\psi(g_k, g_0,\cdots, g_{k-1}).
\end{split}
\]
The two cochain complexes $C^\bullet(\sfG; L)$ and $D^\bullet_{\textup{diff}}(G;L)$ are isomorphism under the following maps
\[
\begin{split}
D^k_{\textup{diff}}(\sfG;L)\ni \psi \mapsto &\ \tilde{\psi}(g_1,\cdots,g_k):=\psi( g_1g_2\cdots g_k, g_2\cdots g_k, ..., g_k, 1_{t(g_k)})\\
C^k_\textup{diff}(\sfG;L)\ni \varphi\mapsto &\ \bar{\varphi}(g_0,\cdots, g_k):=\alpha_{g_k^{-1}}\varphi(g_0g_1^{-1},...,g_{k-1}g_k^{-1}).
\end{split}
\]

We define a characteristic map $\chi: C^\bullet_\textup{diff} (\sfG;L) \cong D^\bullet_{\textup{diff}}(\sfG;L)\to X^\bullet(\calA)$ by
\begin{equation}\label{eq:chi-twist}
\begin{split}
\chi(\varphi)(a_0, \cdots, a_k)&:=\int_M\Big( \int_{g_0\cdots g_k=1}\varphi(g_1, \cdots, g_k)a_0(g_0)a_1(g_1)\cdots a_k(g_k)\Big)\\
\chi(\psi)(a_0, \cdots, a_k)&:=\int_{M\ni x}\Big(\int_{\sfG^k_x} \psi(x,g_1,\cdots, g_k )\tilde{a}_0(x,g_1)\cdots\tilde{a}_k(g_k,x)\Big)
\end{split}
\end{equation}
It is straightforward to check that $\chi$ is a chain map and therefore defines a map on the associated cohomology groups $\chi: H^\bullet_{\textup{diff}}(\sfG;L)\to HC^\bullet(\calA)$ and therefore a well defined pairing between the differentiable cohomology of $\sfG$ and the K-theory of $\calA$
\[
\langle\ ,\ \rangle: H^\bullet_{\textup{diff}}(\sfG;L)\times K_\bullet(\calA)\longrightarrow \C.
\]

The localization of the differentiable cohomology to the unit space can be identified with the Lie algebroid cohomology $H^\bullet_{\textup{Lie}}(A; L)$ with the same arguments as is explained in Section \ref{homogeneous}. In this way, we have obtained a well defined pairing between the Lie algebroid cohomology of $A$ and the localized K-theory of $\calA$,
\[
\llangle\ ,\ \rrangle: H^\bullet_{\textup{Lie}}(A;L)\times K^\textup{loc}_0(\calA)\longrightarrow \C.
\]

Recall that $\pi$ is the projection from $A^*$ to $M$, and $\pi^! A$ is a Lie algebroid over $A^*$, the induced pullback of $A$.
Denote by $L_{\pi^!A}$  the line bundle $\bigwedge^\textup{top}T^*A^*\otimes \bigwedge^\textup{top}\pi^!A$. Note that the pullback 
line bundle $\pi^*L_{A}$ is isomorphic to $\Theta^r\otimes L_{\pi^! A}$ as $\pi^!A$-representations. Since $\Theta^r$ is $\pi^A$-invariant, 
one concludes that $L_{\pi^!A}$ is isomorphic to $\pi^*A$ as $\pi^!A$-representations. 
By \cite[Thm.~2]{crainic}, we obtain the following proposition.
\begin{proposition}\label{cor:alg-coh-coef}
Under the isomorphism between $H^\bullet_{\textup{Lie}}(\pi^!A,\C)$ and $H^\bullet_{\textup{Lie}}(A,\C)$ defined in Cor.~\ref{cor:alg-coh},  
the cohomology groups $H^\bullet_{\textup{Lie}}(\pi^!A; L_{\pi^!A})$ and $H^\bullet_{\textup{Lie}}(A;L_A)$ are naturally isomorphic as $H^\bullet_{\textup{Lie}}(A,\C)$-modules. 
\end{proposition}

There is a canonical pairing between the spaces $H^k_{\textup{Lie}}(\pi^!A;L_{\pi^!A})$ and $H^{2r-k}_{\textup{Lie},c}(\pi^!A;\C)$ given by 
\[
\langle\alpha, \beta \rangle_{\pi^!A}=\int_{A^*} \alpha \wedge \beta. 
\]
In the above equation, $\alpha$ is an element of $\Omega^k(\pi^!A; L_{\pi^!A})$, and $\beta$ is an element of $\Omega^{2r-k}_c(\pi^!A)$, 
and $\alpha\wedge \beta$ is a section of $\bigwedge^{2r}\pi^!A^*\otimes \bigwedge^{2r}\pi^!A \otimes \bigwedge^\textup{top}T^*A^*$. 
With the natural pairing map
\[
 {\bigwedge}^{2r}\pi^!A^*\otimes {\bigwedge}^{2r}\pi^!A \longrightarrow \R,
\]
we can view $\alpha\wedge\beta$ as a compactly supported section of $\bigwedge^\textup{top} T^*A^*$ and integrate it over $A^*$ to define the pairing. Now we are ready to introduce a generalization of the map $\Phi:\Omega^\bullet_{\textup{Lie}}(A;L_A)\to \text{Tot}^\bullet({\mathcal {B}}\calC(\calA^\hbar_{A^*,\textup{c}} (A^*) ))$ using Eq.~(\ref{ind-map}).

For the convenience of computation, we from now on fix a smooth nonvanishing section $\Omega$ of 
$L$. We remark that since $\sfG$ is not assumed to be unimodular,  the section $\Omega$ may not 
$\sfG$-invariant. This defect is measured by the modular function 
$\delta^\Omega_\sfG\in \calC^\infty(\sfG)$ defined by the property 
$ g(\Omega_{s(g)})= \delta^\Omega_\sfG(g)\Omega_{t(g)}$ (cf.~Sec.~\ref{sec:unitr}).  
The homogeneous version 
$D^\bullet_{\textup{diff}}(\sfG;L)$ of the groupoid cohomology can be identified with the space 
$D^\bullet_{\text{AS}, \delta}(\sfG)$ of families of Alexander-Spanier cochains over the $t$-fibers of 
$\sfG\to M$ satisfying the twisted invariance property,
\[
  g^*(\varphi_{s(g)})=\delta^\Omega_\sfG(g^{-1})\varphi_{t(g)},\quad 
  \varphi\in \calC^\infty(\sfG^{\times_{\! t} k}).
\]
With $\Omega$, the characteristic map $\chi$ defined in Eq.~(\ref{eq:chi-twist}) becomes 
identical to the characteristic map $\chi_\Omega$ defined in Eq.~(\ref{eq:chi-uni}), and the 
pairing formula $\langle\ , \ \rangle_{\pi^!A}$ becomes identical to the one in 
Section \ref{sec:alg-ind} with the unimodular assumption.  Furthermore, the same formula as 
Eq.~(\ref{eq:frakX}) defines a natural map $\mathfrak {X}$.  

It is straightforward to check that the proofs of Props.~\ref{comp-trace} - \ref{prop:frakX} 
are independent of the property that $\Omega$ is $\sfG$-invariant, though $\tr_\Omega$ 
and $\tau_\Omega$ are not traces on $\calA^\hbar_{A^*}$ anymore. Therefore, they still hold in the 
present setup. As a conclusion, the proof of Theorem~\ref{main-theorem} explained in 
Section \ref{sec:proof} directly generalizes to give the following result.
\begin{theorem}
\label{gen-loc-ind}
Let $D$ be a longitudinal elliptic differential operator and assume $\alpha\in H^{2k}_\textup{Lie}(A;L_A)$. 
Then the following identity holds true;
\[
 \llangle [\Ind(D)]_\textup{loc},\alpha\rrangle=\frac{1}{(2\pi\sqrt{-1})^k}
 \int_{A^*}\pi^*\alpha\wedge\hat{A}(\pi^!A)\wedge\rho^*_{\pi^!A}{\rm ch}(\sigma(D)),
\]
where we have used the pairing 
$\bigwedge^{2r}\pi^!A^*\otimes \bigwedge ^{2r}\pi^!A \longrightarrow \R$ in the above integral 
formula. 
\end{theorem}
\begin{remark}
Our proof of Theorem B (and Theorem \ref{main-theorem}) relies crucially on the assumption that the Lie algebroid $A$ is integrated to a Lie groupoid $\sfG$. In general, not every Lie algebroid can be integrated to a Lie groupoid, cf.\ \cite{cf}, so this raises the question to extend Theorem B to a pure Lie algebroid statement without the assumption of integrability. 
\end{remark}
\iffalse
\begin{remark}
Quite strikingly, Theorem \ref{main-theorem} is a pure Lie algebroid statement. However, our proof uses the groupoid $\sfG$ inducing this Lie algebroid. Not every Lie algebroid can be integrated to a Lie groupoid, cf.\ \cite{cf}, so this raises the question if integrability is necessary as an assumption.  However, any Lie algebroid is integrable to a {\em local} groupoid: this is similar to a groupoid with the exception that multiplication of arrows is only defined on a neighbourhood $U$ of $M$ in $\sfG$. 
For such local groupoids, a pseudodifferential calculus $\Psi^\infty_{\rm loc}(\sfG)$ is developed in \cite{nwx}. An element $P\in\Psi^\infty_{\rm loc}(\sfG)$ is still a family of pseudodifferential operators $P_x$ on the fibers $t^{-1}(x)$ for $x\in M$, but the difficult issue is to make sense of the property of being $\sfG$-invariant. We can no longer identify the ideal of smoothing operators $\Psi^{-\infty}_{\rm loc}(\sfG)$ with the convolution algebra since this doesn't exist for a local groupoid. However, if we represent Lie algebroid cohomology classes by localized cochains in $U_k$ (c.f.\ Eq. \eqref{ind-nghbd}) in $D^\bullet_{\rm diff}(\sfG;L)$, i.e., invariant Alexander--Spanier cochains, the pairing between elliptic operators and Lie algebroid cohomology classes makes perfect sense. After that we can follow the proof to get the theorem for arbitrary Lie algebroids.
\end{remark}
\fi

\section{Examples}
\label{sec:example}

\subsection{The pair groupoid}
\label{pair}
Let $M$ be a compact manifold. The {\em pair groupoid} has objects given by $M$ and arrows $\sfG_M=M\times M$ so that the source and target maps are given by the two projections. The multiplication is simply given by $(x,y)\cdot (y,z)=(x,z)$ for $x,y,z\in M$. This defines a proper Lie groupoid, so that by \cite[Thm ]{crainic} we have
\begin{equation}
\label{pair-gpd-coh}
H^\bullet_\textup{diff}(\sfG_M;\C)=\begin{cases} \C,&\text{if }\bullet=0,\\ 0,&\text{else}.\end{cases}
\end{equation}
The Lie algebroid of the pair groupoid $\sfG_M$ is given by the tangent bundle $TM$ of $M$, so that the line bundle $L_{TM}$ is canonically trivial. The standard Chevalley--Eilenberg complex for Lie algebroid cohomology identifies with the de Rham complex $(\Omega^\bullet(M),d_\textup{dR})$ of $M$ so that $H^\bullet_\textup{Lie}(TM;\C)\cong H^\bullet(M;\C)$. On the other hand, the complex \eqref{loc-cochain} using localized groupoid cochains also has a very natural interpretation: it is exactly the Alexander--Spanier complex computing the cohomology of a manifold.

The universal enveloping algebra of $TM$ is identified with the algebra $\calD(M)$, the algebra of differential operators on $M$, and with this the Localized Index Theorem \ref{gen-loc-ind} reduces to the localized index theorem proved by Connes--Moscovici \cite{conmos}: for an Alexander--Spanier cocycle $f=f_0\otimes\ldots\otimes f_{2k}\in \calC^\infty(M^{2k+1})$, the pairing with an elliptic differential operator $D$ is given by
\[
\llangle\Ind(D)_\textup{loc},[f]\rrangle=\frac{1}{(2\pi\sqrt{-1})^k}\int_{T^*M}f_0df_1\wedge\ldots\wedge df_k\wedge \hat{A}(T^*M)\wedge {\rm ch}(\sigma(D)).
\]
Because of \eqref{pair-gpd-coh}, there is only one global index, namely the usual one, which under the van Est map is computed to agree with the Atiyah--Singer formula corresponding to $k=0,~f_0=1$ in the expression above.

However, our interpretation in terms of the van Est map makes it natural to look for groupoids that have the same Lie algebroid, but with a richer groupoid cohomology than \eqref{pair-gpd-coh}. For this, consider a discrete group $\Gamma$ acting on a manifold $\tilde{M}$ such that $\tilde{M}\slash\Gamma\cong M$. The universal situation is of course $\Gamma=\pi_1(M)$ acting on the universal covering. Define the Lie groupoid $\sfG_{\Gamma,M}$ with objects given by $M$ and morphisms $\tilde{M}\times_\Gamma\tilde{M}$. This Lie groupoid has the same Lie algebroid $TM$ and since it is Morita equivalent to $\Gamma$, we have 
\begin{equation}
\label{cov-gp}
H^\bullet_\textup{diff}(\sfG_{\Gamma,M};\C)\cong H^\bullet_{\rm grp}(\Gamma;\C)\cong H^\bullet(B\Gamma;\C),
\end{equation}
where $B\Gamma$ denotes the classifying space of $\Gamma$.
The covering space $\tilde{M}$ is classified by a map $\psi:M\to B\Gamma$, unique up to homotopy. 
\begin{lemma}
\label{lemma-cov}
Under the isomorphism \eqref{cov-gp}, the map 
\[
\psi^*: H^\bullet(B\Gamma;\C)\to H^\bullet(M;\C)
\]
identifies with the van Est map $\Phi$.
\end{lemma}
Combined with the two main theorems of this paper, we therefore find exactly the covering index theorem of Connes-Moscovici:
\begin{theorem}[cf.\ \cite{conmos}]
Let $\Gamma$ be a discrete group acting on a manifold $\tilde{M}$ with quotient $\tilde{M}\slash\Gamma\cong M$. Let $D$ be a differential operator on $M$ and denote by $\tilde{D}$ its lift to $\tilde{M}$. For $\alpha\in H^{2k}_{\rm grp}(\Gamma;\C)$ we have
\[
\left<\Ind_\Gamma(\tilde{D}),\alpha\right>=\frac{1}{(2\pi\sqrt{-1})^k}\int_{T^*M}\psi^*\alpha\wedge\hat{A}(T^*M)\wedge {\rm ch}(\sigma(D))
\]
\end{theorem}
\subsection{Foliations}
Let $M$ be a compact manifold with a regular foliation $\calF$. The {\em 
holonomy} (resp.~\emph{monodromy}) \emph{foliation groupoid} $\sfG^h_\calF\rightrightarrows M$ 
(resp.~$\sfG^m _\calF\rightrightarrows M$) consists of holonomy (resp.~homotopy) 
equivalence classes of paths in the same leaf of $\calF$.  In the case 
that $\calF$ is the whole tangent bundle of $M$, $\sfG^h$ is the pair 
groupoid $M\times M$ and $\sfG^m_M$ is the fundamental groupoid 
$\tilde{M}\times_\Gamma \tilde{M}$ considered in the previous subsection.

The Lie algebroid $A$ associated to $\sfG^h$ and $\sfG^m$ are the same and 
is the foliation $\calF$ over $M$. The dual $\calF^*$  of the Lie 
algebroid is a regular Poisson manifold, and $\pi^! \calF$ is the 
symplectic foliation on $\calF^*$.  The bundle $L$ is equal to 
$\bigwedge^{\rm top}\calF\otimes \bigwedge^{\rm top}T^*M$. According to 
\cite[Example 3.11]{elw}, $\calF$ is unimodular if and only if there is a 
transversely invariant measure $\nu$ on $M$. The Lie algebroid cohomology 
of $\calF$ is usually called the foliation cohomology of $\calF$. A 
leafwise elliptic operator $D$ lifts to defines a $\sfG^h$ ($\sfG^m$) 
elliptic operator. Theorem \ref{gen-loc-ind} now gives:
\begin{theorem}
\label{eq:foliation}
Let $\calF\subset TM$ be a foliation on $M$ and $D$ a leafwise elliptic differential operator on $M$. 
For a foliation cocycle $f\in \bigwedge^{2k}\calF^*\otimes L$ we have
\begin{equation*}
\llangle [\Ind(D)]_\textup{loc},f\rrangle=\frac{1}{(2\pi\sqrt{-1})^k}\int_{\calF^*}\pi^*f 
\wedge\hat{A}(\pi^!\calF)\wedge\rho^*_{\pi^!\calF}{\rm ch}(\sigma(D)).
\end{equation*}
When $\calF$ is unimodular, the integral on the the right hand side integral
can be computed on $M$:
\[
\llangle [\Ind(D)]_\textup{loc},f\rrangle=\frac{1}{(2\pi\sqrt{-1})^k}\int_{M}f 
\wedge\hat{A}(\calF^*)\wedge {\rm ch}_\calF(\sigma(D))\wedge \nu,
\]
where $\hat{A}(\calF^*)$, and  ${\rm ch}_\calF(\sigma(D))$ are the leafwise 
characteristic classes of $\calF^*$ and $\sigma(D)$ in 
$H^\bullet_{Lie}(\calF; \C)$.
\end{theorem}
When $f$ is the constant function 1, on the 
holonomy foliation groupoid $\sfG^h_\calF$ we have recovered the 
Connes--Skandalis \cite{CoSk} foliation index theorem for transversely 
invariant measure
\begin{align*}
\tr_\nu([\Ind(D)])&=\left<[\Ind(D)]_\textup{loc},1\right>\\
&=\frac{1}{(2\pi\sqrt{-1})^k}\int_{M}\hat{A}(\calF^*)\wedge 
{\rm ch}_\calF(\sigma(D))\wedge \nu.
\end{align*}

Connes' foliation index theorem in general computes the pairing between 
the K-theory of the holonomy foliation groupoid algebra and its cyclic 
cohomology. Such a pairing does not tell much information about the 
topology of leaves, as the cyclic cohomology of the holonomy foliation 
groupoid algebra is computed \cite{cm} to be equal to the cohomology of 
the leaf space. Our localized index formula of Theorem \ref{eq:foliation} generalizes 
the Connes--Skandalis foliation index theorem for transversely invariant 
measure to detect more leafwise information of a longitudinal elliptic 
operator due to the existence of the cycle $f$ in Theorem \ref{eq:foliation}.

\begin{remark}
The cohomology of the classifying space of $\sfG$ is computed by the 
hypercohomology of the de Rham differential forms on $B\sfG$ with $\sfG$ 
being the etale groupoid associated to a
complete transversal. Connes \cite{connes} computed the pairing between the 
cohomology class of $B\sfG$ (as a component of the cyclic cohomology) with 
the K-theory of the holonomy (monodromy) foliation algebra.

The differentiable cohomology of $\sfG$ is a subcomplex of $\Gamma(B\sfG; 
\Omega^\bullet\otimes L)$ identified as $\Gamma(B\sfG, \Omega^\textup{top}\otimes 
L)$. This gives a natural map from the
differentiable cohomology of $\sfG$ to the cohomology of $B\sfG$, and 
therefore the cyclic cohomology of the groupoid algebra. Our global 
groupoid index theorem provided in Theorem \ref{eq:foliation} computes the pairing between these cyclic 
cohomology classes and the K-group of the holonomy (monodromy) groupoid 
algebra.
\end{remark}

\end{document}

%% file: intro.tex
\section*{Introduction}
In \cite{CoSk}, Connes and Skandalis proved a milestone result in noncommutative geometry and index theory. 
They developed a powerful machinery to prove a far reaching index theorem for leafwise elliptic operator on 
a regular foliation of a compact manifold. This theorem has inspired numerous works in 
foliation theory and noncommutative geometry, as for example \cite{DoHuKa,G,HM}. 
In \cite[Chap.~II, Sec.~10.$\varepsilon$]{connes}, Connes  raised the index problem for general Lie groupoids as a generalization of this foliation index theorem, as well as the Atiyah--Singer Covering Index Theorem and Connes--Moscovici's Homogeneous Space Index Theorem. This has been a long standing open problem in noncommutative geometry. In this paper, we present some progress towards understanding Connes' conjecture.

Let us briefly recall the setting of this conjecture. Let $\sfG\rightrightarrows M$ be a Lie groupoid with Lie algebroid $A$.
Given an elliptic  $\sfG$-differential operator $D$ along the fibers of the source map, Connes defines its analytic index class $\Ind_{a}(D)$ in the $K$-theory of the convolution $C^*$-algebra $C^*(\sfG)$ when $M$ is compact. There are several constructions of this index class. One construction \cite{Carrillo-Rouse} uses Connes' tangent groupoid $\sfG^T$ which combines the Lie algebroid $A$ and $\sfG$ into one single groupoid. This induces a long exact sequence in $K$-theory which can be used to transfer the symbol class $\sigma(D)$ in $K^0(A^*)$ to $K_0(C^*(\sfG))$. This is exactly the analytic index class of $D$. For the pair groupoid of a manifold $M$ this class is the Fredholm index of the elliptic operator $D$ on $M$, for which the Atiyah--Singer index theorem gives a topological expression. In the case of a foliation groupoid, the 
Connes--Skandalis foliation index theorem mentioned above equates this class with a topological index class $\Ind_t(D)$. For a general Lie groupoid, there are many natural interesting examples of elliptic differential operators on $\sfG$, e.g. the leafwise signature operator. To find a topological construction for the analytic index class for a general Lie groupoid is a problem that is wide open since the 1980's.

A strategy to better understand this index class is to construct it in the $K$-theory of the {\em smooth} convolution algebra $\calA:= \calC^\infty_\textup{c} (\sfG)$, and compose with the Chern--Connes character to cyclic homology. This leads to studying  the pairing of the index class with the cyclic cohomology of $\calA$, which results in a number in $\C$, rather than an abstract class in $K$-theory, and the goal is to find a topological expression for this number. Although cyclic cohomology is much more computable than $K$-theory in general, a major obstruction to studying this pairing is that very little is known about the cyclic cohomology of the smooth convolution algebra of a general Lie groupoid.

In this paper, we construct explicit cyclic cocycles associated to differentiable groupoid cocycles, and compute the resulting pairing with the index class in terms of characteristic classes in the Lie algebroid cohomology of $A$. Let us explain in more detail what we prove. Throughout this paper, we assume to work with Hausdorff Lie groupoids with compact unit spaces. 

Our construction of cyclic cocycles on $\calA$ is inspired by the recent development of Hopf cyclic cohomology introduced by Connes and Moscovici \cite{conmos:hopf,conmos:algebroid}. 
A {\em Hopf algebroid} is a noncommutative version of a groupoid. In \cite{conmos:algebroid}, Connes and Moscovici introduced a characteristic map associated to an action of a Hopf algebroid on an algebra satisfying certain properties. By dualizing all the structure maps on the groupoid $\sfG$, the commutative algebra $\calC^\infty(\sfG)$ morally has the structure of a Hopf algebroid. The fact that this is only morally true has to do with  topological issues with tensor products, but this is irrelevant for our application. In our case, $\calC^\infty(\sfG)$ acts on the convolution algebra $\calA$ by simple pointwise multiplication, and the Connes-Moscovici  characteristic map leads to introduce the following characteristic map
\[
  \chi:\bigoplus_{i\geq 0}H^{\bullet+2i}_{\rm diff}(\sfG;L)\to HC^\bullet(\calA),
\]
where $L=\bigwedge^{\rm top}T^*M\otimes\bigwedge ^{\rm top}A$ is the line bundle of ``transverse densities'' \cite{elw} of the Lie algebroid $A$ equipped with a natural $\sfG$ action.  At this point, one does not have to worry anymore about completions and Hopf algebroids. 

We compute the pairing $\langle \chi(\alpha),\Ch(\Ind(D)) \rangle$ by localizing it to the unit space. We construct the index class $\Ind(D)$ in such a way that it has the property that it can be represented by idempotents in the convolution algebra that have support in an arbitrarily small neighbourhood of the unit space of $\sfG$. In fact, we construct in this paper a localized $K$-theory  $K_0^\textup{loc}(\calA)$ of the convolution algebra $\calA$  of a Lie groupoid. The
localized K-theory  of $\calA$ consists of homotopy classes of idempotents in $\calA$ supported in an arbitrarily small neighborhood of the unit space of $\sfG$. This generalizes the definition of Moscovici--Wu \cite{mw} for a manifold or, 
from our point of view, actually the pair groupoid to Lie groupoids. An elliptic $\sfG$-pseudodifferential operator $D$ defines a localized $K$-theory class $\Ind_{\textup{loc}}(D)$ in $K_0^\textup{loc}(\calA)$, which  maps to $\Ind(D)$ along the natural map $K_0^\textup{loc}(\calA)\rightarrow K_0(\calA)$. Analogously, the localization of the differentiable groupoid cohomology to the unit space is the Lie algebroid cohomology $H^\textup{ev}_{\rm Lie}(A;L)$. We remark that after localization to the unit space, the Hopf algebroid is well-defined on the level of jets at the unit space, cf. [KP], which accounts for the appearance of Lie algebroid cohomology. With this, we construct a ``localized index pairing"
\[
\llangle \ , \ \rrangle:\ H^\textup{ ev}_{\rm Lie}(A;L)\times K^\textup{loc}_0(\calA)\to\C.
\] 

In \cite{wx}, the authors constructed a natural ``van Est map'' $E:H^{\bullet}_{\rm diff}(\sfG;L)\to H^\bullet_{\rm Lie}(A;L)$ relating differentiable groupoid and Lie algebroid cohomology. It is exactly this map that controls the relation between the ``localized'' and the ``global'' index pairing:
\begin{thma} 
{\em
Let $\sfG$ be a Lie groupoid.  For the pairing with the Chern-Connes character of the index class of an elliptic $\sfG$-differential operator, one has the equality
\[
 \langle \chi(\alpha),\Ch(\Ind(D))\rangle=\llangle  E(\alpha),[\Ind(D)]_\textup{loc}\rrangle .
\]
}
\end{thma}

We accomplish the computation of the localized index pairing entirely in Lie algebroid terms:
\begin{thmb}
\label{loc-index}
{\em 
Let $A\to M$ be an integrable Lie algebroid and $E$ a vector bundle over $M$.
An elliptic element $D\in\calU(A)\otimes\End(E)$ defines a canonical class $[\Ind(D)]_\textup{loc}\in K_0^\textup{loc}(\calA)$, and for $\alpha\in H^{2k}_{\rm Lie}(A;L)$ we have
\[
\llangle \alpha , [\Ind(D)]_\textup{loc} \rrangle=\frac{1}{(2\pi\sqrt{-1})^k}\int_{A^*}\pi^*\alpha\wedge\hat{A}(\pi^!A)\wedge\rho^*_{\pi^!A}{\rm ch}(\sigma(D)).
\]
}
\end{thmb}
Here, in the topological expression on the right hand side, we use $\pi^!A$, the pull-back Lie algebroid of $A$ along the projection $\pi:A^*\to M$. The individual terms are certain Lie algebroid classes, analogous to the usual $\hat{A}$-genus and the Chern character, and they combine to a compactly supported volume form over $A^*$, which can be integrated. When $A=TM$, the tangent bundle, we find exactly Connes-Moscovici's localized index theorem \cite{conmos}. 

We prove Theorem B using the $\sfG$-pseudodifferential calculus of \cite{montpier} and \cite{nwx}. Its associated symbol calculus defines a quantization of the canonical Lie--Poisson bracket on $A^*$. The index pairing is therefore computed by an algebraic index theorem for formal deformation quantizations of the Poisson manifold $A^*$. Unfortunately, we can not simply apply the algebraic index theorem for Poisson manifolds \cite{tt,cfw,dr}, since an arbitrary formal deformation quantization does not reflect the geometry of the groupoid $\sfG$. We therefore adapt our earlier work \cite{ppt} to prove a $\sfG$-invariant algebraic index theorem on a regular Poisson manifold over $A^*$, which generalizes the foliation index theorem by Nest and Tsygan \cite{nt:family}.

This paper is organized as follows:
We recall in Section \ref{sec:char} the definition of differentiable groupoid cohomology and introduce the characteristic class map from the 
differentiable groupoid cohomology to the cyclic cohomology of the smooth groupoid algebra $\calA$. In Section \ref{sec:localization}, we 
define the localized K-homology of the groupoid $\sfG$, the Lie algebroid cohomology of $A$, and the localized index pairing. Via the Van 
Est map, we relate the index pairing between K-theory and differentiable cohomology of $\sfG$ and the localized index pairing between the 
localized K-theory and the Lie algebroid cohomology of $\sfG$. In Section \ref{sec:local-pairing}, we define the localized index of an elliptic 
$\sfG$ differential operator and introduce the asymptotic $\sfG$-pseudodifferential calculus to compute the localized index. The algebraic 
index theorem of deformation quantization of $A^*$ was developed in \ref{sec:formality}. In Section \ref{sec:local-index}, we prove the main 
theorem of the paper in the case of a unimodular groupoid. The extension of the proof to the general Hausdorff groupoid is explained in 
Section \ref{sec:nonunimodular}.  In Section \ref{sec:example}, we explain our theorem in two examples, the pair groupoid and the foliation 
groupoid. \\

\noindent{\bf Acknowledgments:} We would like to thank Paulo Carrillo-Rouse for inspiring discussion of index formula for elliptic $\sfG$ 
pseudodifferential operators. Pflaum is partially supported by NSF grant DMS 1105670, and 
Tang is partially supported by NSF grant DMS 0900985. Tang acknowledges support from the NWO-cluster ``GQT'' for a visit to the University of Amsterdam. Both X.T. and H.P. thank Marius Crainic for a very helpful discussion during that visit.